\documentclass[12pt]{amsart}
\usepackage{calc,amssymb,amsthm,amsmath, ulem, fullpage}
\usepackage{alltt}

\RequirePackage[dvipsnames,usenames]{color}  

\RequirePackage{hyperref}
\hypersetup{
bookmarksdepth=3,
bookmarksopen,
bookmarksnumbered,
pdfstartview=FitH,
colorlinks,bookmarks,backref,hyperindex,
linkcolor=Sepia,
anchorcolor=BurntOrange,
citecolor=Blue,
filecolor=BlueViolet,
menucolor=Yellow,
urlcolor=OliveGreen
}

\DeclareFontFamily{OMS}{rsfs}{\skewchar\font'60}
\DeclareFontShape{OMS}{rsfs}{m}{n}{<-5>rsfs5 <5-7>rsfs7 <7->rsfs10 }{}
\DeclareSymbolFont{rsfs}{OMS}{rsfs}{m}{n}
\DeclareSymbolFontAlphabet{\scr}{rsfs}

\newtheorem{theorem}{Theorem}[section]
\newtheorem*{mainthm}{Main Theorem}
\newtheorem*{thm*}{Theorem}
\newtheorem{lemma}[theorem]{Lemma}
\newtheorem{proposition}[theorem]{Proposition}
\newtheorem{corollary}[theorem]{Corollary}

\theoremstyle{definition}
\newtheorem{definition}[theorem]{Definition}

\theoremstyle{remark}
\newtheorem{remark}[theorem]{Remark}

\newtheorem{question}[theorem]{Question}

\newcommand{\tld}{\widetilde }

\newcommand{\blank}{\underline{\hskip 10pt}}

\newcommand{\sH}{\scr{H}}

\newcommand{\sL}{\scr{L}}
\newcommand{\sM}{\scr{M}}

\newcommand{\mJ}{\mathcal{J}}

\newcommand{\ba}{\mathfrak{a}}

\newcommand{\bm}{\mathfrak{m}}
\newcommand{\bA}{\mathbb{A}}

\newcommand{\bQ}{\mathbb{Q}}

\renewcommand{\O}{\scr O}

\newcommand{\N}{\mathbb {N}}
\newcommand{\R}{\mathbb {R}}

\newcommand{\Q}{\mathbb {Q}}

\renewcommand{\O}{\mbox{$\mathcal{O}$}}

\newcommand{\infinity}{\infty}

\newcommand{\tensor}{\otimes}

\DeclareMathOperator{\Ann}{{Ann}}

\DeclareMathOperator{\id}{{id}}

\DeclareMathOperator{\supp}{{supp}}

\DeclareMathOperator{\Spec}{{Spec}}

\DeclareMathOperator{\Hom}{Hom}
\DeclareMathOperator{\sHom}{{\sH}om}

\def\spec#1.#2.{{\bold S\bold p\bold e\bold c}_{#1}#2}
\def\proj#1.#2.{{\bold P\bold r\bold o\bold j}_{#1}\sum #2}
\def\ring#1.{\scr O_{#1}}
\def\map#1.#2.{#1 \to #2}
\def\longmap#1.#2.{#1 \longrightarrow #2}
\def\factor#1.#2.{\left. \raise 2pt\hbox{$#1$} \right/
\hskip -2pt\raise -2pt\hbox{$#2$}}
\def\pe#1.{\mathbb P(#1)}
\def\pr#1.{\mathbb P^{#1}}

\def\coh#1.#2.#3.{H^{#1}(#2,#3)}
\def\dimcoh#1.#2.#3.{h^{#1}(#2,#3)}
\def\hypcoh#1.#2.#3.{\mathbb H_{\vphantom{l}}^{#1}(#2,#3)}
\def\loccoh#1.#2.#3.#4.{H^{#1}_{#2}(#3,#4)}
\def\dimloccoh#1.#2.#3.#4.{h^{#1}_{#2}(#3,#4)}
\def\lochypcoh#1.#2.#3.#4.{\mathbb H^{#1}_{#2}(#3,#4)}
\def\ses#1.#2.#3.{0  \longrightarrow  #1   \longrightarrow
 #2 \longrightarrow #3 \longrightarrow 0}
\def\sesshort#1.#2.#3.{0
 \rightarrow #1 \rightarrow #2 \rightarrow #3 \rightarrow 0}
 
\def\iff#1#2#3{
    \hfil\hbox{\hsize =#1 \vtop{\noin #2} \hskip.5cm
    \lower.5\baselineskip\hbox{$\Leftrightarrow$}\hskip.5cm
    \vtop{\noin #3}}\hfil\medskip}
\def\myoplus#1.#2.{\underset #1 \to {\overset #2 \to \oplus}}

\usepackage[all]{xy}

\theoremstyle{definition}
\newtheorem{definition-proposition}[theorem]{Definition-Proposition}

\newcommand{\cO}{\mathcal{O}}
\renewcommand{\O}{\mathcal{O}}
\newcommand{\fra}{\mathfrak{a}}

\renewcommand{\to}[1][]{\xrightarrow{\ #1\ }}

\renewcommand{\th}{\ensuremath{^\text{th}}}
\renewcommand{\frm}{\frak{m}}
\newcommand{\frp}{\frak{p}}
\newcommand{\defeq}{\stackrel{\scriptscriptstyle \operatorname{def}}{=}}

\begin{document}

\title{Discreteness and rationality of $F$-jumping numbers on singular varieties}
\author{Manuel Blickle, Karl Schwede, Shunsuke Takagi and Wenliang Zhang}

\address{Fakult\"at f\"ur Mathematik \\ Universit\"at Duisburg-Essen (Campus Essen) \\ Universit\"atsstr.~2 \\ 45117 Essen, Germany}
\email{manuel.blickle@gmail.com}
\address{Department of Mathematics\\ University of Michigan\\ East Hall
530 Church Street \\ Ann~Arbor, Michigan, 48109, USA}
\email{kschwede@umich.edu}
\address{Department of Mathematics\\ Kyushu University, 6-10-1\\ Hakozaki, Higashi-ku\\
Fukuoka 812-8581, Japan}
\email{stakagi@math.kyushu-u.ac.jp}
\address{Department of Mathematics\\ University of Michigan\\ East Hall
530 Church Street \\ Ann~Arbor, Michigan, 48109, USA}
\email{wlzhang@umich.edu}


\thanks{The first author is supported by the Deutsche Forschungsgemeinschaft (DFG) through a Heisenberg fellowship and through the SFB/Transregio 45 \emph{Periods, moduli spaces and arithmetic of algebraic varieties}}
\thanks{The second author was partially supported by a National Science Foundation postdoctoral fellowship and by RTG grant number 0502170.}
\thanks{The third author was partially supported by Grant-in-Aid for Young Scientists (B) 20740019 from JSPS and by Program for Improvement of Research Environment for Young Researchers from SCF commissioned by MEXT of Japan.}
\thanks{The fourth author is partially supported by a Spring/Summer 2009 Research Fellowship from Department of Mathematics, University of Michigan.}

\begin{abstract}
We prove that the $F$-jumping numbers of the test ideal $\tau(X; \Delta, \ba^t)$ are discrete and rational under the assumptions that $X$ is a normal and $F$-finite scheme over a field of positive characteristic $p$, $K_X+\Delta$ is $\bQ$-Cartier of index not divisible $p$, and either $X$ is essentially of finite type over a field or the sheaf of ideals $\ba$ is locally principal. This is the largest generality for which discreteness and rationality are known for the jumping numbers of multiplier ideals in characteristic zero.
\end{abstract}
\maketitle
\numberwithin{equation}{theorem}
\section{Introduction}

For a normal $\bQ$-Gorenstein variety $X$, an ideal sheaf $\ba \subseteq \O_X$, and a positive real number $t > 0$ the multiplier ideal $\mJ(X, \ba^t)$ is an ideal sheaf of $\O_X$ that reflects subtle properties of both, the singularities of $X$,   as well as the singularities of elements of $\ba$. In recent years the multiplier ideal has become a central notion in algebraic geometry, as it is a key tool in the birational classification of varieties.
As $t$ increases, the multiplier ideal becomes a smaller (deeper) ideal. Due to the definition of the multiplier ideal in terms of a single resolution of singularities (for all $t$), the values of $t$ where the multiplier ideal changes form a discrete set of rational numbers, called the \emph{jumping numbers} of $(X,\ba)$, see \cite{EinLazSmithVarJumpingCoeffs}. The first jumping number is called the \emph{log canonical threshold} which is an important invariant in the study of birational geometry; see for example \cite{KollarSingularitiesOfPairs}. More generally, the sequence of jumping numbers contains a great deal of information about the singularities of $X$ and elements of $\ba$, see \cite{EinLazSmithVarJumpingCoeffs}, \cite{SaitoMultIdealsBFunctionAndSpectrum}, \cite{BudurMustataSaitoBernsteinSatoPolynomialsOfArbitrary}, \cite{TuckerJumpingNumbersOnRationalSurfaces}.

In characteristic $p > 0$, the test ideal $\tau(X,\ba^t)$ is an analog of the multiplier ideal which enjoys many similar properties (see for example \cite{HaraYoshidaGeneralizationOfTightClosure} and \cite{TakagiFormulasForMultiplierIdeals}), but also differs in some key ways (see \cite{MustataYoshidaTestIdealVsMultiplierIdeals}).
Test ideals are rooted in the theory of tight closure of Hochster and Huneke \cite{HochsterHunekeTC1}, see also \cite{SmithTestIdeals} and \cite{HaraInterpretation}. The variants we are considering here were introduced in \cite{HaraYoshidaGeneralizationOfTightClosure} and \cite{TakagiInterpretationOfMultiplierIdeals}. Just as for multiplier ideals, as one increases $t$, the test ideal $\tau(X, \ba^t)$ becomes a smaller ideal.  Therefore, it is natural to define the \emph{$F$-jumping numbers} of the pair $(X,\ba)$ as those $\alpha$ such that $\tau(X,\ba^{\alpha-\epsilon})$ \emph{strictly} contains $\tau(X,\ba^{\alpha})$ for all $\epsilon > 0$. The set of $F$-jumping numbers, and in particular the first of these, called the $F$-pure threshold which is the analog of the log canonical threshold from characteristic zero, have been studied for example in \cite{TakagiWatanabeFPureThresh}, \cite{MustataTakagiWatanabeFThresholdsAndBernsteinSato},  \cite{HunekeMustataTakagiWatanabeFThresholdsTightClosureIntClosureMultBounds} and \cite{MustataBernsteinSatoPolynomialsInPositive}. Already in \cite{MustataTakagiWatanabeFThresholdsAndBernsteinSato} it was observed that the $F$-jumping numbers share many properties with the jumping numbers of the multiplier ideal, however, one of the most basic properties, namely their discreteness and rationality was left open as a question. It is not at all clear from the definition whether discreteness and rationality hold.  Likewise, the fact that under reduction modulo $p$, the multiplier ideal $\mJ(X, \ba^t)$  agrees with the test ideal $\tau(X, \ba^t)$ for $p \gg 0$ (see \cite{SmithMultiplierTestIdeals}, \cite{HaraInterpretation}, \cite{HaraYoshidaGeneralizationOfTightClosure}, and \cite{TakagiInterpretationOfMultiplierIdeals}) does not help with this question since the value of $t$ may increase the particular $p \gg 0$ you need to use.

However, since the work of the first author with Smith and Musta\c{t}\v{a} \cite{BlickleMustataSmithDiscretenessAndRationalityOfFThresholds} and \cite{BlickleMustataSmithFThresholdsOfHypersurfaces} (see also Hara and Monsky's \cite{HaraMonskyFPureThresholdsAndFJumpingExponents}), one knows that discreteness and rationality of $F$-jumping numbers holds for pairs $(X,\ba^t)$ when $X$ is smooth and $F$-finite and either $X$ is essentially of finite type over a field or $\ba$ is a principal ideal. The techniques used in these previous articles heavily relied on the smoothness of the ambient space $X$ as either the flatness of the Frobenius or $D$-module techniques entered in an essential way. The few known results in the singular case are quite restrictive, see for example \cite{TakagiTakahashiDModulesOverRingsWithFFRT} and \cite{SchwedeTakagiRationalPairs}. In  \cite{KatzmanLyubeznikZhangOnDiscretenessAndRationality} a new proof of the discreteness result for principal ideals in regular local rings was given which seemed less dependent on smoothness.  This slightly extended (out of the $F$-finite case) the result in \cite{BlickleMustataSmithFThresholdsOfHypersurfaces} (see also \cite{Blickle.MinimalGamma}).  This gave hope for a proof in the general singular case, at least for $\ba$ principal\footnote{Partial work in this direction was done in \cite{SchwedeTakagiRationalPairs} where it was shown that the result in \cite{KatzmanLyubeznikZhangOnDiscretenessAndRationality} extends to the case of a Gorenstein strongly $F$-regular ring.}.  Indeed, in this paper we accomplish this (and much more) as we show the discreteness and rationality for $F$-jumping numbers under the same general hypotheses where discreteness and rationality are known for jumping numbers of the multiplier ideal in characteristic zero. Our main result is as follows (see Theorems \ref{ThmMainNonPrincipalCase} and \ref{ThmFFinitePrincipalCase} below for the most general statements).
\begin{mainthm}
    Let $X$ be an $F$-finite scheme over a field of characteristic $p > 0$. Assume that $X$ is normal and $\bQ$-Gorenstein with index not divisible by $p$. Let $\ba$ be a nonzero sheaf of ideals of $\O_X$.
    \begin{enumerate}
        \item If $X$ is in addition \emph{essentially of finite type over a field}, or
        \item $\ba$ is a \emph{principal ideal},
    \end{enumerate}
    then the $F$-jumping numbers of the pair $(X,\ba)$ are a discrete set of rational numbers.
\end{mainthm}
In fact, we show this result in the more general log $\bQ$-Gorenstein setting. That is, we remove the $\bQ$-Gorenstein hypothesis but additionally take an effective $\bQ$-divisor $\Delta$ and assume that $K_X + \Delta$ is $\bQ$-Cartier with index not divisible by $p$.  We then show that discreteness and rationality hold for the jumping numbers of the test ideal $\tau_b(X; \Delta, \ba^t)$ associated to such log $\bQ$-Gorenstein triples (see Definition-Proposition \ref{DefnPropBigTestIdeals}). This generalization is not arbitrary but essential to the proof of part (1) of the Main Theorem. It allows us -- and this is the key new ingredient here -- to use the second author's theory of $F$-adjunction \cite{SchwedeFAdjunction} (reviewed below in Section \ref{SectionQDivisorsAndFSingularities}) to translate the question of discreteness for the test ideals $\tau_b(X,\Delta,\ba^t)$ to a question about yet another variant of test ideals $\tau_b(Y,\nsubseteq Q,\Delta_Y,\ba^t)$ where $Y=\bA^m_k$ is affine $n$-space and $Q$ is a prime ideal of $k[x_1, \ldots, x_n]$. The key point is that the ambient space is affine $n$ space.  Therefore, the methods of \cite{BlickleMustataSmithDiscretenessAndRationalityOfFThresholds}, and in particular the argument bounding the degree of the generators of test ideals, can be adapted to derive discreteness and rationality for these more general test ideals $\tau_b(Y,\nsubseteq Q,\Delta_Y,\ba^t)$; this is carried out in Section \ref{SectionDiscRatKalg}.

Besides in the work of the second author \cite{SchwedeFAdjunction} these test ideals, or variants thereof, have appeared in the work of the third author in \cite{TakagiPLTAdjoint}, \cite{TakagiHigherDimensionalAdjoint}. Roughly speaking, $\tau_b(Y = \Spec S, \nsubseteq Q; \Delta, \ba^t)$ is like the usual test ideal $\tau(Y = \Spec S, \Delta, \ba^t)$ except we restrict the ``test elements'' to the multiplicative set $S \setminus Q$ instead of $S^{\circ}$. The necessary results on this are given in Section \ref{SectionRestrictionOfTestIdeals} which is preceded by an overview to various definitions and characterizations of test ideals in Section \ref{SectionReviewOfGeneralized}. We should emphasize at this point that our viewpoint and most of our techniques rely on the interpretation of the test ideal in terms of $p^{-e}$-linear maps, that is additive maps $\phi_e: R \to R$ such that $\phi_e(r^{p^e}s)=r\phi_e(s)$. If we consider, for simplicity, pairs $(R,\ba^t)$ where $R$ is an $F$-finite domain, then the test ideal $\tau(R,\ba^{t})$ can be characterized as the unique smallest ideal $J$ such that for all $p^{-e}$-linear maps $\phi_e$ one has $\phi_e(\ba^{\lceil t(p^e-1) \rceil}J) \subseteq J$, see Proposition \ref{propCharTestSmallestFixed}.

To prove the second part of the Main Theorem we cannot use the degree argument of the finite type case. As alluded to above, we instead follow the method of the proof in the smooth case of \cite{KatzmanLyubeznikZhangOnDiscretenessAndRationality}, or more precisely, we use a globalized dual version of it. This is carried out in Section \ref{SectionDiscretenessAndRationalityForPrincipal}. A key ingredient of this proof is a result on the stabilization of iterated images of $p^{-e}$-linear maps on finitely generated modules, which may be viewed as a (globalized) dual to a celebrated result of Hartshorne and Speiser
\cite[Proposition 1.11]{HartshorneSpeiserLocalCohomologyInCharacteristicP} (which was generalized by Lyubeznik in \cite{LyubeznikFModulesApplicationsToLocalCohomology}) on the stabilization of kernels of iterated $p^e$-linear maps on co-finite modules. The statement is as follows (cf.~Theorem~\ref{TheoremDualHSLForFFinite}): If $M$ is a coherent $\O_X$-module equipped with a $p^{-e}$-linear map $\phi: M \to M$, then the sequence of images
\[
    M \supseteq \phi(M) \supseteq \phi^2(M) \supseteq \phi^3(M) \ldots
\]
eventually stabilizes. This result (in the $F$-finite case) can be derived from a result of the first author in \cite{Blickle.MinimalGamma}, and in its most general form appeared in \cite[Lemma 13.1]{Gabber.tStruc}.

There are a number of ways in which one may hope to generalize the rationality and discreteness results for jumping numbers of test ideals we obtained here -- some of which are addressed in Section \ref{SectionFurtherRemarks}.  The most pressing of these is probably the question if rationality or discreteness holds without the $\bQ$-Gorenstein hypothesis. Note that even in characteristic zero for the multiplier ideals -- which were recently defined in \cite{DeFernexHacon} in the non $\bQ$-Gorenstein case -- the discreteness and rationality of their jumping numbers are not known.

\subsubsection*{\it Acknowledgments:}
The authors would like to thank Daniel Hernandez for several valuable discussions.  The second and third authors discussed some of these ideas at the \emph{Mathematical Sciences Research Institute} (MSRI) in Berkeley during the winter of 2009.  The first and second authors discussed some of these ideas at the conference on Multiplier Ideals held at the \emph{Mathematisches Forschungsinstitut Oberwolfach} in April of 2009.

\section{Notation and preliminaries on \texorpdfstring{$F$}{F}-singularities}
\label{SectionQDivisorsAndFSingularities}

Throughout this paper, all schemes $X$ will be assumed to be noetherian, separated, and $F$-finite (see Definition \ref{DefnFFinite} below) over a field of positive characteristic $p$. Since we often work with divisors we typically assume that $X$ is also \emph{normal} (see Section \ref{SectionFurtherRemarks} for comments on the non-normal case).

The Frobenius (or $p$\th{} power) map will play a fundamental role in this paper. For any scheme $X$ of characteristic $p > 0$, the absolute Frobenius $X \to[F] X$ is a morphism of schemes defined as the identity map on the underlying topological space of $X$, but the corresponding map of structures sheaves $\mathcal O_X \rightarrow F_*\mathcal O_X$ sends sections to their $p$\th~powers. If $X=\Spec R$ is affine, then the Frobenius $F: R \to R$ is hence given by $F(r)=r^p$ for all $r \in R$. We will often consider the iterated composition of $F$ with itself to obtain the $e\th{}$ power of the Frobenius $F^e$ which is given on sections by the $(p^e)\th$ power map.

\begin{definition}
\label{DefnFFinite}
A scheme $X$ of prime characteristic $p$ is \emph{$F$-finite} if the Frobenius is a finite map, i.e.\  $F_*\cO_X$ is a finitely generated $\cO_X$-module. A ring $R$ is called $F$-finite if $\Spec R$ is $F$-finite.
\end{definition}

Given an $R$-module $M$, we will use $F^e_* M$ to denote the $R$-module which agrees with $M$ as an additive group, but where the multiplication is defined by $r.m = r^{p^e} m$.  This notation is justified since, if $X = \Spec R$ and $\sM$ is the quasi-coherent $\O_X$-module corresponding to $M$, then the push-forward $F^e_* \sM$ is the quasi-coherent sheaf corresponding to $F^e_* M$.

By Gabber \cite[Remark 13.6]{Gabber.tStruc} an affine $F$-finite scheme $X$ possesses a normalized dualizing complex. Its cohomology in degree $-\dim X$ is a canonical module $\omega_X$ for $X$. In general, one knows that $F^{e!}\omega_X := \Hom_{\O_X}(F^e_* \O_X, \omega_X)$ is also a canonical module. Hence, by a uniqueness statement for dualizing complexes in \cite[Theorem V.3.1]{HartshorneResidues} there is a line bundle $\sM$ such that
\[
    F^{e!}\omega_X \cong F^e_*(\omega_X \tensor \sM).
\]
Hence on an affine cover trivializing $\sM$ we have for each member an isomorphism $F^{e!}\omega_X \cong F^e_*\omega_X$ of $F^e_*\O_X$-modules. Summarizing, this shows that any $F$-finite scheme $Y$ has a finite affine cover such that each member $X$ of the cover satisfies the following conditions:
\begin{equation}\label{EqnAdditionalAssumption}\tag{$\dagger$}
\begin{cases}
\text{$X$ has a canonical module $\omega_X$, and} \\
\Hom_{\O_X}(F^e_* \O_X, \omega_X) = (F^e)^! \omega_X \cong F^e_* \omega_X \text{ as $F^e_*\O_X$-modules.}
\end{cases}
\end{equation}
By Proposition \ref{PropPassToAffineCover} below, the proofs of our main results in this paper can be carried out on any finite affine cover. Therefore it poses no restriction to our main results (but eases notation somewhat) to assume throughout the paper that all schemes are affine and satisfy condition (\ref{EqnAdditionalAssumption}) -- hence we will do so from now on. Also note that if $X$ is essentially of finite type over an $F$-finite field, or if $X$ is the spectrum of a local ring, condition (\ref{EqnAdditionalAssumption}) is always satisfied, see \cite[V.10]{HartshorneResidues}. Furthermore, it may be that condition (\ref{EqnAdditionalAssumption}) holds for every $F$-finite noetherian scheme (although we do not know a proof).

\begin{definition}
\label{DefPMinuseLinear}
Given $\O_X$-modules $M$ and $N$, an additive map $\phi : M \rightarrow N$ is called \emph{$p^{-e}$-linear} or \emph{Cartier linear} if for all local sections $r\in \O_X$ and $m \in M$ the relation $\phi(r^{p^e} m) = r \phi(m)$ is satisfied. In other words, a $p^{-e}$-linear map is an $\O_X$-linear map $F^e_* M \to N$ (also denoted by $\phi$).

If $X=\Spec R$ with $R$ reduced and $M = N = R$, a $p^{-e}$-linear map can be viewed as a map $R^{1 \over p^e} \rightarrow R$ where $R^{\frac{1}{p^e}}$ is the purely inseparable extension of $R$ consisting of all $(p^e)$\th{} roots of elements of $R$.
\end{definition}

\begin{remark}
\label{RemarkDefinitionOfComposedPInverseLinear}
A $p^{-e}$-linear map $L \rightarrow M$ and a $p^{-d}$-linear map $M \rightarrow N$ clearly compose to a $p^{-e-d}$-linear map $L \rightarrow N$.  Equivalently, viewing the first map as a map $F^e_* L \rightarrow M$, we can apply the functor $F^d_*$ to obtain $F^{e+d}_* L \rightarrow F^d_* M$ which we compose with the map $F^d_* M \rightarrow N$ to obtain a map $F^{e+d}_* L \to M$. In this way we can compose a $p^{-e}$-linear endomorphism $\phi_e : M \rightarrow M$ with itself to obtain $p^{-ne}$-linear endomorphisms $\phi^n_e : M \rightarrow M$ for all $n > 0$.  We will use $\phi_{ne} : F^{ne}_* M \rightarrow M$ to denote the $R$-linear map corresponding to $\phi_{e}^n$.

Sometimes we will have a $\O_X$-linear map $\phi_e : F^e_* \sL \rightarrow \O_X$ where $\sL$ is a line bundle.  In such a case, we can twist by $\sL$ (using the projection formula) to obtain a map
\[
    F^e_*\sL^{1+p^e} \cong F^e_*(\sL \tensor F^{e*}\sL) \cong F^e_*\sL \tensor \sL \to[\phi_e\tensor \sL] \sL
\]
Pushing forward by $F^e_*$ and composing with $\phi_e$ gives us a map $\phi_{2e} : F^e_* \sL^{1 + p^e} \rightarrow \O_X$.  Repeating this process, we obtain maps $\phi_{ne} : F^{ne}_* \sL^{1 + p^e + \ldots + p^{(n-1)e}} \rightarrow \O_X$.
\end{remark}

\begin{definition}
 A \emph{pair $(X, \Delta)$} is the combined information of a normal separated integral scheme $X$ and an effective $\bQ$-divisor on $X$.  A \emph{triple $(X, \Delta, \ba^t)$} is the combined information of a pair $(X, \Delta)$, a non-negative real number $t \geq 0$, and a non-zero ideal sheaf $\ba$ on $X$.
\end{definition}

For characteristic zero triples and pairs, it is often also assumed that $K_X + \Delta$ is $\bQ$-Cartier (that is, some power is locally trivial in the divisor class group).  In characteristic $p > 0$ this assumption is unnecessary in order to \emph{define} the basic notions. In their study however, the condition that $K_X + \Delta$ is $\bQ$-Cartier with \emph{index not divisible by $p > 0$} is often quite useful as we will see shortly


We recall properties of pairs $(X,\Delta)$ in the \emph{log $\bQ$-Gorenstein setting}. That is, we assume that $K_X + \Delta$ is $\bQ$-Cartier. In fact we assume slightly more, namely that $K_X + \Delta$ is $\bQ$-Cartier with \emph{index not divisible by the characteristic $p$}. This is clearly equivalent to assuming that for some $e \geq 0$ the divisor $(p^e-1)(K_X+\Delta)$ is a Cartier divisor.

In this context, $F$-singularities of pairs were explored in the second author's paper \cite{SchwedeFAdjunction} and most of what follows in this section is taken from there. Similar techniques have however appeared before; see for example \cite{MehtaRamanathanFrobeniusSplittingAndCohomologyVanishing} or \cite[Proof \#2 of Theorem 3.3]{HaraWatanabeFRegFPure}.

At the heart of this treatment of $F$-singularities of pairs is a correspondence between such divisors $\Delta$ and certain $p^{-e}$-linear maps. More precisely, let $X$ denote a normal $F$-finite integral scheme, then there is a bijection of sets:
\begin{equation}
\label{EqnGlobalBijection}\tag{$\star$}
\left\{ \begin{matrix}\text{Effective $\bQ$-divisors $\Delta$ on $X$ such}\\\text{that $(p^e - 1)(K_X + \Delta)$ is Cartier}\end{matrix} \right\} \leftrightarrow \left\{ \begin{matrix}\text{Line bundles $\sL$ and non-zero }\\ \text{elements of $\Hom_{\O_X}(F^e_* \sL, \O_X)$} \end{matrix}\right\} \Big/ \sim
\end{equation}
The equivalence relation on the right side identifies two maps $\phi_1: F^e_* \sL_1 \rightarrow \cO_X$ and $\phi_2 : F^e_* \sL_2 \rightarrow \cO_X$ if there is an isomorphism $\gamma : \sL_1 \rightarrow \sL_2$ and a commutative diagram
\begin{equation*}
\begin{split}
\xymatrix{
 F^e_* \sL_1 \ar[d]_{F^e_* \gamma} \ar[r]^{\phi_1} & \cO_X \ar[d]^{\id} \\
 F^e_* \sL_2 \ar[r]^{\phi_2} & \cO_X \\
}
\end{split}
\end{equation*}
A crucial ingredient in the correspondence (\ref{EqnGlobalBijection}) is Grothendieck duality for the finite morphism $F^e$. We sketch the construction and refer the reader to \cite{SchwedeFAdjunction} for details. Since $F^{e!}\O_X \defeq \sHom_{\O_X}(F^e_*\O_X,\O_X) \cong \O_X((1-p^e)K_X)$ Grothendieck duality \cite[Chapter III, Section 6]{HartshorneResidues}  yields an isomorphism of $F^e_* \O_X$-modules
\[
    F_*^e\sHom_{\O_X}(\sL,\O_X((1-p^e)K_X)) \cong \sHom_{\O_X}(F^e_*\sL,\O_X)
\]
for all $\O_X$-modules $\sL$. Suppose now that we are given $\Delta$ such that $(p^e - 1) (K_X + \Delta)$ is Cartier. With $\sL = \O_X((1-p^e)(K_X + \Delta))$ we obtain from the above isomorphism an inclusion:
\begin{align*}
 \O_X & \subseteq F^e_* \O_X \subseteq F^e_* \O_X((p^e-1)\Delta) \cong F^e_* \sHom_{\O_X}(\O_X((1-p^e)\Delta),\O_X)\\
 & \cong  F^e_* \sHom_{\O_X}(\O_X((1-p^e)(K_X + \Delta)), \O_X((1-p^e)K_X))\\
 & \cong F^e_*\sHom_{\O_X}(\sL, \O_X((1-p^e)K_X)) \cong \sHom_{\O_X}(F^e_*\sL, \O_X)
\end{align*}
The map $\phi$ associated with $\Delta$ is now the global section of $\sHom_{\O_X}(F^e_* \sL, \O_X)$ corresponding to the global section $1$ of $\O_X$ via this inclusion.

Conversely, an element $\phi \in \sHom_{\O_X}(F^e_* \sL, \O_X) \cong F^e_* \sL^{-1}( (1-p^e) K_X)$ determines an $F^e_*\O_X$-linear map
\[
    F^e_*\O_X\to[1\mapsto \phi]\sHom_{\O_X}(F^e_* \sL, \O_X) \to[\sim] F^e_* \sL^{-1}( (1-p^e) K_X)
\]
which corresponds to an effective Weil divisor $D$ such that $\O_X(D) \cong \sL^{-1}( (1-p^e) K_X)$.  Set $\Delta = {1 \over p^e - 1}D$.

In the special case that $X$ is the Spectrum of a local ring $(R, \bm)$, then the   correspondence (\ref{EqnGlobalBijection}) can be interpreted as follows:
\[
\left\{ \begin{matrix}\text{Effective $\bQ$-divisors $\Delta$ such }\\\text{that $(p^e - 1)(K_X + \Delta)$ is Cartier}\end{matrix} \right\} \leftrightarrow \left\{ \text{Non-zero elements of $\Hom_{R}(F^e_* R, R)$} \right\} \Big/ \sim
\]
where two maps $\phi_1, \phi_2 \in \Hom_{R}(F^e_* R, R)$ are identified if there exists some unit $u \in R$ such that $\phi_1(x) = \phi_2(ux)$ (in other words, if $\phi_1$ and $\phi_2$ agree after pre-multiplication by a unit).
Finally, if $X$ is the $\Spec$ of a complete local ring $(R, \bm)$, then this correspondence can also be extended to include the following:
\[
 \left\{ \begin{matrix} \text{Nonzero $R\{F^e\}$-module} \\ \text{structures on $E_R$} \end{matrix} \right\} \Big/ \sim
\]
Here the equivalence relation identifies two maps if they are equal after post-multiplication by a unit of $R$.

\begin{remark}
Suppose that $X = \Spec R$ and that $\Delta$ is an effective $\bQ$-divisor such that $(p^e - 1)\Delta$ is integral.  Further suppose that $\Hom_{\O_X}(F^e_* \O_X((p^e - 1)\Delta), \O_X)$ is a free $F^e_* \O_X$-module.  Then the $\phi_e$ corresponding to $\Delta$ can be viewed as a map $\phi_e : F^e_* R \rightarrow R$.  To see this, simply note that we have isomorphisms as $F^e_* \O_X$-modules,
\[
F^e_* \O_X \cong \Hom_{\O_X}(F^e_* \O_X((p^e - 1)\Delta), \O_X) \cong F^e_* \O_X( (1-p^e)(K_X + \Delta)).
\]
Therefore, the $\sL$ corresponding to $\Delta$ via (\ref{EqnGlobalBijection}) is isomorphic to $\O_X$.  Alternately, since we have a natural inclusion $\xymatrix{\Hom_{\O_X}(F^e_* \O_X((p^e - 1)\Delta), \O_X) \ar@{^{(}->}[r] & \Hom_{\O_X}(F^e_* \O_X, \O_X),}$ $\phi_e$ can be interpreted as the image of a generator of $\Hom_{\O_X}(F^e_* \O_X((p^e - 1)\Delta), \O_X)$ inside $\Hom_{\O_X}(F^e_* \O_X, \O_X)$.
\end{remark}

We now briefly review how singularities in characteristic $p > 0$ can be defined using such pairs.  For the purposes of this paper, one can take the following as definitions.

\begin{proposition} \cite[Theorem 3.11, Proposition 4.1, Proposition 4.8]{SchwedeFAdjunction}
\label{PropFSingularitiesInTermsOfMaps}
 Suppose that $X$ is a normal integral scheme and that $(X, \Delta)$ is a pair such that $K_X + \Delta$ is $\bQ$-Cartier with index not divisible by $p > 0$.  Let $\phi : F^e_* \sL \rightarrow \O_X$ be an associated map.  Then:
\begin{itemize}
 \item[(a)]  $(X, \Delta)$ is \emph{sharply $F$-pure} if and only if $\phi$ is surjective (as a map of sheaves).
 \item[(b)]  A reduced and irreducible subscheme $W \subseteq X$ is a \emph{center of $F$-purity for $(X, \Delta)$} if and only if $\phi(F^e_* I_W \sL) \subseteq I_W$ (here $I_W$ is the ideal sheaf of $W$).
 \item[(c)]  $(X, \Delta)$ is \emph{strongly $F$-regular} if and only if $(X, \Delta)$ has no centers of $F$-purity.
\end{itemize}
\end{proposition}

The true value of the observations made at the beginning of the section is that it allows the transfer of a problem on a ring $R$ to a quotient ring $R/Q$ (or visa versa) where $Q$ corresponds to some center of $F$-purity for $(X, \Delta)$.  The following result summarizes some ways in which this can be done.

\begin{theorem} \cite[Main Theorem]{SchwedeFAdjunction}
\label{TheoremFAdjunction}
Suppose that $X$ is an integral normal $F$-finite noetherian scheme of essentially finite type over a field of characteristic $p > 0$.  Further suppose that $\Delta$ is an effective $\bQ$-divisor on $X$ such that $K_X + \Delta$ is $\bQ$-Cartier with index not divisible by $p$.  Let $W \subseteq X$ be an closed subscheme that satisfies the following properties:
\begin{itemize}
\item[(a)]  $W$ is integral and normal.
\item[(b)]  $(X, \Delta)$ is sharply $F$-pure at the generic point of $W$.
\item[(c)]  $W$ is a center of $F$-purity for $(X, \Delta)$.
\end{itemize}
Then there exists a canonically determined effective divisor $\Delta_W$ on $W$ satisfying the following properties:
\begin{itemize}
\item[(i)]  $(K_W + \Delta_W) \sim_{\bQ} (K_X + \Delta)|_W$
\item[(ii)]  The singularities of $(W, \Delta_W)$ are ``the same'' as the singularities of $(X, \Delta)$ near $W$.  (Please see \cite{SchwedeFAdjunction} for details of what ``the same'' means).
\end{itemize}
\end{theorem}
\begin{proof}
 We will sketch the proof of (i).  Let $\phi : F^e_* \sL \rightarrow \O_X$ be a map corresponding to $\Delta$ via (\ref{EqnGlobalBijection}).  By hypothesis, $\phi(F^e_* I_W \sL) \subseteq I_W$.  This implies that we have a diagram,
\[
 \xymatrix{
F^e_* I_W \sL \ar@{^{(}->}[d] \ar[r]^{\phi} & I_W \ar@{^{(}->}[d] \\
F^e_* \sL \ar@{->>}[d] \ar[r]^{\phi} & \O_X \ar@{->>}[d] \\
F^e_* \sL|_W \ar[r]^{\overline{\phi}} & \O_W
}
\]
Applying (\ref{EqnGlobalBijection}) on $W$, we obtain an effective $\bQ$-divisor $\Delta_W$ on $W$ such that $\O_W((1-p^e)(K_W + \Delta_W)) \cong \sL_W$.  To see that $\Delta_W$ is independent of the choice of $\phi$ (and in particular, independent of the choice of $e$) see \cite[Theorem 3.10, Theorem 5.2]{SchwedeFAdjunction}.
\end{proof}

When one combines this result with Fedder's work, see \cite{FedderFPureRat}, one obtains the following result.

\begin{theorem} \cite[Theorem A, Theorem 5.5]{SchwedeFAdjunction}
\label{TheoremFedderFadjunction}
Suppose that $S$ is a regular $F$-finite ring such that $F^e_* S$ is free and that $R = S/I$ is a quotient that is normal.  Further suppose that $\Delta_R$ is an effective $\bQ$-divisor on $X = \Spec R$ such that $(p^e - 1)\Delta$ is integral and $\O_X((p^e - 1)(K_X + \Delta_R))$ is free.  Then there exists an effective $\bQ$-divisor $\Delta_S$ on $S$ such that:
\begin{itemize}
 \item[(i)]  $X$ is a center of $F$-purity for $(\Spec S, \Delta_S)$.
 \item[(ii)]  $\Delta_S$ and $\Delta_R$ are related as in Theorem \ref{TheoremFAdjunction}.
\end{itemize}
\end{theorem}

This theorem will allow us to translate the techniques \cite{BlickleMustataSmithDiscretenessAndRationalityOfFThresholds} from the polynomial ring case to the case of a quotient of a polynomial ring. Roughly speaking, if a ring $R=S/I$ is a quotient of a polynomial ring $S$, then questions about $(R, \Delta_R)$ can be answered by studying the pair $(S, \Delta_S)$.

\section{Generalized test ideals}
\label{SectionGenTest}
This section contains the technical bulk of this article, as we collect and derive here the definitions and results about generalized test ideals we need to prove our main result on rationality and discreteness of $F$-jumping numbers. We begin in Section \ref{SectionReviewOfGeneralized} by recalling the various notions and definitions of generalized test ideals that have appeared in the literature before, and show that these are equivalent in the context we are considering. This is mostly expository and serves the purpose to collect results which are scattered through the literature. Section \ref{SectionRestrictionOfTestIdeals} follows an extension of this notion of generalized test ideal which is crucial in our treatment, namely that of a \emph{generalized test ideal along an $F$-pure center}. This construction -- building on ideas going back to \cite{HochsterHunekeTC1}, \cite{TakagiHigherDimensionalAdjoint}, and \cite{TakagiPLTAdjoint} -- slightly reformulates work of the second author in \cite{SchwedeFAdjunction} and allows us to make the transition from the case of affine space to the case of an arbitrary normal affine variety. We finish this section with some general results on these test ideals which will be needed in the proof of both cases of our main theorem.

\subsection{Review of generalized test ideals}
\label{SectionReviewOfGeneralized}
We first set up some notation for the rest of this section. $R$ will denote a normal domain of characteristic $p > 0$, and $\Delta$ will denote an effective $\R$-divisor on $X:=\Spec R$. By $R^\circ$ we denote the set of non zero-divisors of $R$. For any integral Weil divisor $D$ on $X$, we will use $R(D)$ to denote the global sections of $\O_X(D)$.

We let $M$ be an $R$-module. For each integer $e \in \N$, we denote $\mathbb{F}^{e,\Delta}(M)=\mathbb{F}^{e,\Delta}_R(M):= F^e_* R(\lceil (p^e-1)\Delta \rceil) \otimes_{R} M$ and regard it as an $R$-module by the action of $F^e_* R \cong R$ from the left. Then the $e\th$ iteration of the Frobenius map induces a map $F^e: M \to \mathbb{F}^{e,\Delta}(M)$. The image of $z \in M$ via this map is denoted by $z^{p^e} := 1 \tensor z = F^e(z) \in \mathbb{F}^{e,\Delta}(M)$. For an $R$-submodule $N$ of $M$, we denote by $N^{[p^e],\Delta}_M$ the image of the map $\mathbb{F}^{e,\Delta}(N) \to \mathbb{F}^{e,\Delta}(M)$.
In the special case that $I \subseteq R$ is an ideal, then $I^{[p^e],\Delta}_R \cong I^{[p^e]} R(\lceil (p^e-1)\Delta \rceil)$.

We denote by $E =E_R =\bigoplus_{\frm} E_R(R/\frm)$ the direct sum, taken over all maximal ideals $\frm$ of $R$, of the injective hulls of the residue fields $R/\frm$.

\begin{definition}
Suppose that $(X=\Spec R, \Delta, \ba)$ is a triple where $R$ is a normal domain of characteristic $p>0$, and let $t>0$ be a real number.
\begin{enumerate}
\item (cf.~\cite[Definition 6.1]{HaraYoshidaGeneralizationOfTightClosure}, \cite[Definition 2.1]{TakagiInterpretationOfMultiplierIdeals}, \cite{HochsterHunekeTC1}) For $R$-modules $N \subseteq M$ the \textit{$(\Delta, \ba^t)$-tight closure} $N^{*(\Delta, \ba^t)}_M$ of $N$ in $M$ is defined to be the submodule of $M$ consisting of all elements $z \in M$ for which there exists $c \in R^{\circ}$ such that
    \[
        c\ba^{\lceil t(q-1) \rceil}z^q \subseteq N^{[q], \Delta}_M
    \]
    for all large $q = p^e$. The $(\Delta, \ba^t)$-tight closure of an ideal $I \subseteq R$ is defined as $I^{*(\Delta, \ba^t)} = I^{*(\Delta, \ba^t)}_R$.
\item (\cite[Definition 2.5]{SchwedeSharpTestElements})
    We say that a nonzero element $c \in R$ is a \textit{sharp test element} for $(X, \Delta, \ba^t)$ if, for all ideals $I \subseteq R$ and all $z \in I^{*(\Delta, \ba^t)}$, we have
    \[
        c\ba^{\lceil t(q-1) \rceil}z^q \subseteq I^{[q]}R(\lceil (p^e-1) \Delta \rceil)
    \]
    for every $q=p^e$.
\item (\cite[Definition 2.19]{SchwedeCentersOfFPurity}, \cite{HochsterFoundations})
    We say that a nonzero element $c \in R$ is a \textit{big sharp test element} for $(X,\Delta, \ba^t)$ if for all inclusions of $R$-modules $N \subseteq M$ and all $z \in N_{M}^{*\Delta}$, we have
    \[
        c\ba^{\lceil t(q-1) \rceil}z^q \subseteq N^{[q],\Delta}_M
    \]
    for every $q=p^e$.  If $\ba = R$ and $\Delta = 0$, then a big sharp test element will be called simply a \emph{big test element}.
\end{enumerate}
\end{definition}

We now define the generalized test ideals.  The original context in which this was defined was when $\Delta = 0$ and $\ba = R$, see \cite{HochsterHunekeTC1}.  It is called the ``test ideal'' because it is generated by the elements which can be used to ``test'' the inclusion in the tight closure.

\begin{definition-proposition}[cf.~{\cite[Definition-Theorem 6.5]{HaraYoshidaGeneralizationOfTightClosure}, \cite[Theorem 2.8 (1)]{TakagiInterpretationOfMultiplierIdeals}, \cite[Theorem 2.7]{SchwedeSharpTestElements}}] \label{PropDefClassicTestElements}
Suppose that $(X=\Spec R, \Delta, \ba)$ is a triple where $R$ is an excellent normal domain of characteristic $p>0$. Then for all real numbers $t>0$ each of the following conditions defines the same ideal, which is called the \emph{generalized test ideal} for the triple $(X; \Delta, \ba^t)$ and denoted by $\tau(X; \Delta, \ba^t)$.
\begin{enumerate}
\item $\displaystyle\bigcap_{M\subseteq E}\!\!\Ann_R(0^{*(\Delta, \ba^t)}_M)$, where
$M$~runs through all \emph{finitely generated submodules} of~$E_R$.
\item $\displaystyle\bigcap_M \Ann_R(0^{*(\Delta,\ba^t)}_M)$, where $M$ runs
through all \emph{finitely generated} $R$-modules.
\vspace{3pt}
\item $\displaystyle\bigcap_{J\subseteq R} (J:J^{*(\Delta, \ba^t)})$, where $J$
runs through all ideals of $R$.\\
\newcounter{temp}
\setcounter{temp}{\value{enumi}}
\end{enumerate}
If $R$ is $F$-finite and $\Delta=0$, then the above ideals also coincide with the following ideal.
\begin{enumerate}
\setcounter{enumi}{\value{temp}}
\item The ideal generated by all sharp test elements for $(X, \ba^t)$.
\end{enumerate}
\end{definition-proposition}

We hope that condition (4) of Proposition-Definition \ref{PropDefClassicTestElements} can be generalized to the case when $\Delta \neq 0$, but we will not need it and so we will not attempt to work it out here. A better-behaved variant of this definition is obtained by dropping the \emph{finitely generated} assumption. These (non-finitistic test ideals) are the test ideals we are considering in this article (although, as we will see, the two notions agree in cases we consider, see Proposition \ref{PropTwoIdealsAgree}).
\begin{definition-proposition}
\label{DefnPropBigTestIdeals}
Suppose that $(X=\Spec R, \Delta, \ba)$ is a triple where $R$ is an $F$-finite normal domain of characteristic $p>0$. Then for all real numbers $t>0$ each of the following conditions defines the same ideal, which is called the \emph{big} generalized test ideal for the triple $(X; \Delta, \ba^t)$ and denoted by $\tau_b(X; \Delta, \ba^t)$.
\begin{enumerate}
\item $\Ann_R(0^{*(\Delta, \ba^t)}_E)$.
\item $\displaystyle\bigcap_M \Ann_R(0^{*(\Delta,\ba^t)}_M)$, where $M$ runs through \emph{all} $R$-modules.
\item The ideal generated by all big sharp test elements for $(X, \Delta, \ba^t)$.
\item The sum
\[
\sum_{e \ge 0} \sum_{\phi} \phi\left(F^e_*(d\ba^{\lceil t(p^e-1) \rceil}) \right),
\] where $\phi$ ranges over $\phi \in \Hom_{R}(F^e_* R(\lceil (p^e-1)\Delta \rceil), R) \subseteq \Hom_{R}(F^e_*R, R)$ and where $d$ is a big sharp test element for $(X, \Delta, \ba^t)$.
\end{enumerate}
If $\Delta = 0$, then the above four ideals coincide with the following ideal.
\begin{itemize}
\item[(5)] The sum
\[
\sum_{e \ge 0} \sum_{\phi} \phi\left(F^e_*( \ba^{\lceil t p^e \rceil} d ) \right),
\]
where $\phi$ ranges over $\phi \in \Hom_{R}(F^e_* R, R)$ and where $d$ is a big test element for $R$
\end{itemize}
\end{definition-proposition}

In the literature, for example in \cite{LyubeznikSmithCommutationOfTestIdealWithLocalization} or \cite{TakagiPLTAdjoint}, the ideal $\tau_b(X; \Delta, \ba^t)$ is often denoted by $\tld \tau(X; \Delta, \ba^t)$.  The key reason why the big test ideal is better behaved than the usual test ideal, is that its formation is easily seen to commute with localization (via condition (4), since one can find big sharp test elements who remain big sharp test elements after localization).

\begin{remark}
Condition (5) can also be generalized to include the case where $\Delta \neq 0$, and we do something quite like this in Lemma \ref{LemmaAlternateUltraGeneralizedTestIdeal} below.  We'll only prove the more limited version here though.
\end{remark}

\begin{proof}
First we prove the equivalence of (1) through (4).  Let us denote
\[
J:= \sum_{e \ge 0} \sum_{\phi} \phi\left(F^e_*(d\ba^{\lceil t(p^e-1) \rceil}) \right)
\]
where $\phi$ ranges over $\Hom_{R}(F^e_* R(\lceil (p^e-1)\Delta \rceil), R) \subseteq \Hom_{R}(F^e_* R, R)$ and $d$ is a fixed big sharp test element for $(X, \Delta, \ba^t)$.
By \cite[Theorem 2.22]{SchwedeCentersOfFPurity}, $J$ coincides with the ideal generated by all big sharp test elements for $(X, \Delta, \ba^t)$.
Also, by definition, every big sharp test element is in $\Ann_R(0^{*(\Delta,\ba^t)}_M)$ for all $R$-modules $M$, and $\bigcap_M \Ann_R(0^{*(\Delta,\ba^t)}_M)$ is clearly contained in $\Ann_R(0^{*(\Delta, \ba^t)}_E)$.
Thus, for the equivalence of (1) through (4), it suffices to prove that $\Ann_R(0^{*(\Delta, \ba^t)}_E)$ is contained in $J$.

First note that
\[
\Ann_R(0^{*(\Delta, \ba^t)}_{E}) = \Ann_R(\bigoplus_{\bm} 0^{*(\Delta, \ba^t)}_{E_{R_{\bm}}})
= \bigcap_{\bm} \Ann_R(0^{*(\Delta, \ba^t)}_{E_{R_{\bm}}}) = \bigcap_{\bm} \Ann_{R_{\bm}}(0^{*(\Delta, \ba^t)}_{E_{R_{\bm}}})
\]
where $\bm$ runs through all maximal ideals of $R$ and $E_{R_\bm}$ is the injective hull $E_{R}(R/\bm) =E_{R_\bm}(R_{\bm}/\bm R_{\bm})$ of the residue field $R/\bm$.
From this observation, we know that the localization of $\Ann_R(0^{*(\Delta, \ba^t)}_{E})$ at a maximal ideal ${\bm}$ of $R$ is contained in $\Ann_{R_{\bm}}(0^{* (\Delta, \ba^t)}_{E_{R_{\bm}}})$.
On the other hand, it is easy to see that the formation of the ideal $J$ commutes with localization (as long as $d$ is chosen to be a big sharp test element that remains a big sharp test element after localization, see Remark \ref{RemBigTestEltsExist}). It follows from \cite[Lemma 2.1]{HaraTakagiOnAGeneralizationOfTestIdeals}
that $\Ann_{R_{\bm}}(0^{* (\Delta, \ba^t)}_{E_{R_{\bm}}})=J_{\bm}$ for every maximal ideal $\bm$ of $R$. Therefore, we obtain the desired inclusion.

Condition (5), in the case that $R$ is local, is simply \cite[Lemma 2.1]{HaraTakagiOnAGeneralizationOfTestIdeals}.  But then we have an ideal of $R$ which agrees with $\tau_b(R, \ba^t)$ after localization at every maximal ideal of $R$.
\end{proof}

\begin{remark}
Conditions (4) and (5) in Definition-Proposition \ref{DefnPropBigTestIdeals} are really slight reformulations of \cite[Lemma 2.1]{HaraTakagiOnAGeneralizationOfTestIdeals}.
\end{remark}

\begin{remark}\label{RemBigTestEltsExist}
Suppose that $d$ is a big sharp test element that remains a big sharp test element after localization and completion (the usual proof of the existence of test elements produces such elements; in particular such elements exist).  It then follows quickly from Definition-Proposition \ref{DefnPropBigTestIdeals}~(4) above that the big test ideal commutes with localization and completion.  See \cite{HaraTakagiOnAGeneralizationOfTestIdeals} and \cite{SchwedeCentersOfFPurity} for additional details.
\end{remark}



It is known that the test ideal coincides with the big test ideal in many cases (see for example \cite{LyubeznikSmithCommutationOfTestIdealWithLocalization}, \cite{LyubeznikSmithStrongWeakFregularityEquivalentforGraded}, \cite{AberbachMacCrimmonSomeResultsOnTestElements}).  We will need the following globalized version of \cite[Theorem 2.8 (2)]{TakagiInterpretationOfMultiplierIdeals} (also compare with \cite[Definition-Theorem 6.5]{HaraYoshidaGeneralizationOfTightClosure}).

\begin{proposition}
\label{PropTwoIdealsAgree}
Suppose that $(X=\Spec R, \Delta, \ba)$ is a triple where $R$ is an $F$-finite normal domain of characteristic $p>0$.
If $K_X + \Delta$ is $\Q$-Cartier, then the ideals $\tau(X;\Delta, \ba^t)$ and $\tau_b(X;\Delta, \ba^t)$ coincide.
\end{proposition}

\begin{proof}
Let $E =\bigoplus_{\frm} E_R(R/\frm)$ be the direct sum, taken over all maximal ideals $\frm$ of $R$,
of the injective hulls of the residue fields $R/\frm$.
For every finitely generated $R$-submodule $M$ of $E$, there exist finitely many maximal ideals $\frm_1, \dots, \frm_l$ of $R$ and finitely generated $R_{\frm_i}$-submodules $M'_i$ of $E_R(R/\frm_i)=E_{R_{\frm_i}}(R_{\frm_i}/\frm_i R_{\frm_i})$ for all $i=1, \dots, l$ such that $M$ is contained in the $R$-module $M':=\bigoplus_{i=1}^l M'_i$.
In this case,
$$\Ann_R (0^{*(\Delta, \ba^t)}_{M}) \supseteq \Ann_R (0^{*(\Delta, \ba^t)}_{M'})=\Ann_R (\bigoplus_{i=1}^l 0^{*(\Delta, \ba^t)}_{M'_i})=\bigcap_{i=1}^l \Ann_R (0^{*(\Delta, \ba^t)}_{M'_i}).$$
Here note that $M'_i$ is also a finitely generated $R$-module for all $i=1, \dots, l$, because it is supported at the maximal ideal $\frm_i$.
Therefore, by Definition-Proposition \ref{PropDefClassicTestElements}, one has $\tau(X,\Delta; \ba^t)=\bigcap_{\frm} \bigcap_{M \subseteq E_{R_{\frm}}} \Ann_R (0^{*(\Delta, \ba^t)}_{M})$, where $\frm$ runs through all maximal ideals of $R$ and $M$ runs through all finitely generated $R_{\frm}$-submodules of $E_{R_{\frm}}:=E_{R_{\frm}}(R_{\frm}/\frm R_{\frm})$.
Note that the ideal $\Ann_R (0^{*(\Delta, \ba^t)}_{M})$ is contained in $\Ann_{R_{\frm}} (0^{*(\Delta, \ba^t)}_{M})$.  Since we are assuming that $K_X + \Delta$ is $\bQ$-Cartier, by \cite[Definition-Theorem 6.5]{HaraYoshidaGeneralizationOfTightClosure} and \cite[Theorem 2.8 (2)]{TakagiInterpretationOfMultiplierIdeals}, we see that $\bigcap_{M \subseteq E_{R_{\frm}}} \Ann_{R_{\frm}} (0^{*(\Delta, \ba^t)}_{M}) = \Ann_R (0^{*(\Delta, \ba^t)}_{E_{R_{\bm}}})$.  Therefore
$$\tau(X; \Delta, \ba^t) \subseteq \bigcap_{\frm} \Ann_{R_{\frm}} (0^{*(\Delta, \ba^t)}_{E_{R_{\frm}}})=\Ann_R (0^{*(\Delta, \ba^t)}_{E})=\tau_b(X; \Delta, \ba^t).$$
The converse inclusion is obvious.
\end{proof}
\begin{remark}
We point out that by \cite{LyubeznikSmithStrongWeakFregularityEquivalentforGraded} and also \cite[Lemma 3]{BlickleMultiplierIdealsAndModulesOnToric}, the big test ideal and the test ideal also coincide in the case where $\Delta = 0$, $R=\bigoplus_{n \ge 0} R_n$ is an $\N$-graded ring with $R_0$ an $F$-finite field and $\ba$ a homogeneous ideal.
\end{remark}
The following characterization of the test ideal is simply a globalized version of a result of the second author, \cite[Proposition 4.8]{SchwedeFAdjunction}.  Also compare with Definition-Proposition \ref{DefnPropBigTestIdeals} (4) and \cite[Theorem 2.6]{SmithFRatImpliesRat}.

\begin{proposition} \cite[Proposition 4.8]{SchwedeFAdjunction}, \cite{SchwedeCentersOfFPurity}
\label{propCharTestSmallestFixed}
Suppose that $(X=\Spec R, \Delta, \ba)$ is a triple where $R$ is an $F$-finite normal domain of characteristic $p>0$.
Furthermore, suppose that $K_X+\Delta$ is $\Q$-Cartier with index not divisible by $p>0$, and let $\phi_e: F^e_*\sL \to \cO_X$ be the corresponding map for some $e \in \N$.
Then for all real numbers $t>0$, the ideal $\tau_b(X;\Delta, \ba^t)(=\tau(X;\Delta, \ba^t))$ is the unique smallest nonzero ideal $J$ such that
\[
\phi_{ne}(F^{ne}_* \ba^{\lceil t(p^{ne}-1) \rceil} J \sL^{1 + p^e + \ldots + p^{(n-1)e}}) \subseteq J
\]
for all integers $n \geq 0$.
\end{proposition}
We conclude this subsection by relating the test ideals introduced up to now with the description given in the case of a regular ring $R$ in \cite{BlickleMustataSmithDiscretenessAndRationalityOfFThresholds}. There, the test ideal is constructed from an operation on ideals of $R$ which is inverse to the operation of the Frobenius sending an ideal $I$ to the ideal $I^{[p^e]}$. Namely, if $J \subseteq R$ is an ideal one defines $J^{[1/p^e]}$ to be the smallest ideal $I$ such that $I^{[p^e]} \supseteq J$.  Then the test ideal of $(R,\ba^t)$ (and temporarily denoted by $\tau'(R,\ba^t)$) is defined as the stable member of the increasing chain of ideals $(\ba^{\lceil tp^e \rceil})^{[1/p^e]}$:
\begin{equation}
\label{EqnBMSTestIdeal}
    \tau'(R,\ba^t) = (\ba^{\lceil tp^e \rceil})^{[1/p^e]} \text{ for $e \gg 0$.}
\end{equation}
To see that $\tau'(R,\ba^t) = \tau_b(R,\ba^t)$ we will prove the following proposition (which is well known to experts) which makes the connection with Definition-Proposition \ref{DefnPropBigTestIdeals} (5) transparent.  The following was also observed in  \cite[Remark 2.2]{TakagiTakahashiDModulesOverRingsWithFFRT}.  Here, we use the notation $\Hom_R(F^e_*R,R) \cdot F^e_*J$ to denote
\[
\sum_{\phi \in \Hom_R(F^e_*R,R)}(\phi(F^e_* J)).
\]
\begin{proposition}
\label{propInverseFrobSplitting}
Suppose that $R$ is regular and  $F$-finite.  Let $J \subseteq R$ be an ideal. Then
\[
 J^{[1/p^e]} = \Hom_R(F^e_*R,R) \cdot F^e_*J.
\]
Furthermore, if we assume that $\Hom_R(F^e_* R, R) \cong F^e_* R$ (this happens in an affine neighborhood of every point on a regular scheme), then
\[
    J^{[1/p^e]} = \phi_e(F^e_*J)
\]
where $\phi_e$ is a local generator of $\Hom_R(F^e_*R,R)$ as an $F^e_*R$-module.
\end{proposition}
\begin{proof}
A quick proof of this goes as follows: In the proof of \cite[Proposition 2.5]{BlickleMustataSmithDiscretenessAndRationalityOfFThresholds} it is observed that $J^{[1/p^e]}$ is equal to $\Hom_R(F^e_*R,R) \cdot F^e_*J$.   For the second part, since $\Hom_R(F^e_*R,R) = F^e_*R \cdot \phi_e = \phi_e(F^e_*R \cdot \blank)$, and since $F_*^eRF_*^eJ=F_*^eJ$, the result follows.
\end{proof}
Hence, by Definition-Proposition \ref{DefnPropBigTestIdeals} (5) we have that
\[
\begin{split}
        \tau'(R,\ba^t) &= \sum_e (\ba^{\lceil tp^e \rceil})^{[1/p^e]} \\
        &= \sum_e \Hom_R(F^e_*R,R) \cdot F_*^e\ba^{\lceil t p^e \rceil} \\
        &= \tau_b(R,\ba^t)
\end{split}
\]
since we may take $d=1$ as the big sharp test element of the regular ring $R$.

\begin{remark}
\label{remarkCartierMap}
A typical map $\phi_e$ considered in the proposition arises from the Cartier map $C^e: F^e_*\omega_R \to \omega_R$, that is the map dual to the Frobenius under Grothendieck-Serre duality. In the case of the polynomial ring $R=k[x_1,\ldots,x_n]$ over a perfect field $k$, it can be explicitly described as follows. Using multi-index notation $x^{\underline{i}}=x^{i_1}\cdot\ldots\cdot x^{i_n}$ for ${\underline{i}}=(i_1,\ldots,i_n)$ and $dx=dx_1\wedge\ldots\wedge dx_n$ the Cartier map is given by sending the differential form $x^{\underline{i}}\cdot\frac{dx}{x}$ to $x^{{\underline{i}}/p^e}\cdot\frac{dx}{x}$ or zero if one of the exponents is not an integer, i.e.~some $i_j$ is not divisible by $p^e$.

Identifying $R$ with $\omega_R$ by sending $1$ to $dx$, the induced map $\psi_e: F^e_*R \to R$ is then given by
\[
    x^{\underline{i}} \mapsto x^{(({\underline{i}}+\underline{1})/p^e)-\underline{1}}
\]
where $\underline{1}=(1,\ldots,1)$ and the expression on the right hand side is zero if one of the exponents is not an integer. In particular, $x^{{p^e-1}}\defeq x_1^{p^e-1}\cdot\ldots\cdot x_n^{p^e-1}$ is mapped to $1$ and all other elements of the basis $\{e_{e,{\underline{i}}}\defeq x^{\underline{i}}\}_{0\leq i_j \leq p^e-1}$ of $F^e_*R$ as an $R$--module are mapped to zero. Since $x^{{p^e-1}}$ is divisible by all the other elements of that basis, it follows that the maps $\psi_{e,{\underline{i}}}\defeq F^e_*(x^{p^e-1}/x^{\underline{i}})\psi_e$ for $0 \leq i_j \leq p^e-1$ form the $R$-module basis of $\Hom_R(F^e_*R,R)$ dual to $\{e_{e,\underline{i}}\}$, hence $\psi_e$ generates $\Hom_R(F^e_*R,R)$ as an $F^e_*R$-module.
\end{remark}

\subsection{Test ideals along \texorpdfstring{$F$}{F}-pure centers}
\label{SectionRestrictionOfTestIdeals}

In this section we study a variant of test ideals which will play a crucial role in our generalization of the discreteness and rationality proof found in \cite{BlickleMustataSmithDiscretenessAndRationalityOfFThresholds}.  Suppose that $R = S/I$ where $S$ is regular.  The goal of this section is to construct a test ideal theory which, when combined with Theorem \ref{TheoremFedderFadjunction}, allows us to compute the usual test ideal of $R$ in terms of this variant of a test ideal for $S$.  The basic idea is as follows:

Suppose $R$ is a domain.  In the classical tight closure theory, one chooses the elements used to test tight closure from $R \setminus \{ 0\}$.  In this variant, we choose a smaller multiplicative set (note that similar constructions have appeared before in \cite{HochsterHunekeTC1}, \cite{TakagiPLTAdjoint} and \cite{TakagiHigherDimensionalAdjoint}).

The main difficulty with constructing such a variant of a test ideal is the existence of test elements.  Therefore, one must choose the set to replace $R \setminus \{0\}$ with some care.  The basics of this theory was worked out in \cite[Section 6]{SchwedeFAdjunction}.  However, we will need some very minor reformulations for our purposes.

\begin{definition}
\label{DefnUltraGeneralizedTestIdeal}
Suppose that $(X = \Spec R, \Delta, \ba^t)$ is a triple where $\O_X((p^{e_0} - 1)(K_X + \Delta))$ is free and $e_0$ is chosen to be minimal with respect to this condition.  Set $\phi_{e_0} : F^{e_0}_* \O_X \rightarrow \O_X$ to be a map corresponding to $\Delta$.  Further suppose that $W \subseteq \Spec X$ is a normal integral scheme corresponding to a prime ideal $Q \in \Spec R$ and that
\begin{itemize}
 \item[(i)]  $W$ is a center of $F$-purity for $(X, \Delta)$.
 \item[(ii)]  $(X, \Delta)$ is sharply $F$-pure at $Q$, the generic point of $W$.
 \item[(iii)]  $\ba$ does not vanish on $W$.
\end{itemize}
Then $\tau_b(X, \nsubseteq Q; \Delta, \ba^t)$ is defined to be the unique smallest ideal $J$ not contained in $Q$ such that $\phi_{ne_0}(F^{ne_0}_* \ba^{\lceil t(p^{ne_0} - 1) \rceil} J) \subseteq J$ for all $n > 0$.
\end{definition}

\begin{remark}
The fact that $e_0$ was chosen to be the smallest such $e_0$ is not actually a necessary hypothesis.  It is, however, easier to simply assume $e_0$ is minimal than to prove that any choice of $e_0$ gives the same definition.
\end{remark}

\begin{remark}
If $\ba = R$, then it is easy to see that $\tau_b(X, \nsubseteq Q; \Delta, \ba^t) = \tau_b(X, \nsubseteq Q; \Delta)$ is the unique smallest ideal $J$ not contained in $Q$ such that $\phi_{e_0}(F^{e_0}_* J) \subseteq J$ (for just the one map $\phi_{e_0}$).  It then follows that $\phi_{e_0}(F^{e_0}_* \tau_b(X, \nsubseteq Q; \Delta)) = \tau_b(X, \nsubseteq Q; \Delta)$, for if not, then there are only two possibilities:
\begin{itemize}
 \item[(1)]  $\phi_{e_0}(F^{e_0}_* \tau_b(X, \nsubseteq Q; \Delta)) \subseteq Q$.
 \item[(2)]  $\phi_{e_0}(F^{e_0}_* \tau_b(X, \nsubseteq Q; \Delta)) \nsubseteq Q$.
\end{itemize}
If condition (1) holds, then after localizing at $Q$, we obtain that $\phi_{e_0}(F^{e_0}_* R_Q) \subseteq Q R_Q$ which contradicts the sharp $F$-purity of $(X, \Delta)$ at $Q$.  On the other hand, if (2) holds, then since $\phi_{e_0}(F^{e_0}_* \tau_b(X, \nsubseteq Q; \Delta)) \subsetneq \tau_b(X, \nsubseteq Q; \Delta)$, we have contradicted the minimality of $\tau_b(X, \nsubseteq Q; \Delta)$.
\end{remark}

The difficulty related to Definition \ref{DefnUltraGeneralizedTestIdeal} is showing that $\tau_b(X, \nsubseteq Q; \Delta, \ba^t)$ exists.  However, once one has that, one easily obtains the following result (compare with \cite[Theorem 4.4]{TakagiPLTAdjoint}):

\begin{theorem} \cite[Corollary 6.9]{SchwedeFAdjunction}
\label{ThmRestrictionForTestIdeals}
In the context of Definition \ref{DefnUltraGeneralizedTestIdeal}, set $\Delta_W$ to be the $\bQ$-divisor on $W$ corresponding to $\Delta$ via Theorem \ref{TheoremFAdjunction}.  Then $\tau_b(X, \nsubseteq Q; \Delta, \ba^t) |_W = \tau_b(W; \Delta_W, \overline{\ba}^t)$.
\end{theorem}

To show that $\tau_b(X, \nsubseteq Q; \Delta, \ba^t)$ exists one first proves the following lemma.  The proof has similarities to the usual proof of existence of test elements for reduced $F$-finite rings.

\begin{proposition} \cite[Proposition 6.14]{SchwedeFAdjunction}
\label{PropositionUniformExistenceForTestElts}
In the context of Definition \ref{DefnUltraGeneralizedTestIdeal}, there is an element $c \in \ba \cap (R \setminus Q)$ such that, for every $d \in R \setminus Q$, there exists an integer $n_d > 0$ so that $c \in \phi_{n_de_0}(F^{n_de_0}_* d\ba^{\lceil t(p^{n_de_0} - 1) \rceil})$.
\end{proposition}

One can then define $\tau_b(X, \nsubseteq Q; \Delta, \ba^t)$ as follows.

\begin{definition}
\label{DefnAlternateAdjointTestIdeal}
Fix a $c$ as in Proposition \ref{PropositionUniformExistenceForTestElts}.  Then:
\[
 \tau_{b}(R,  \nsubseteq Q; \Delta, \ba^t)= \sum_{n \geq 0} \phi_{ne_0}(F^{ne_0}_*c \ba^{\lceil t(p^{ne_0} - 1) \rceil}).
\]
In particular, $\tau_{b}(R,  \nsubseteq Q; \Delta, \ba^t)$ exists.
Note that the sum stabilizes as a finite sum since $R$ is noetherian.  Further observe that $c$ is contained in $\tau_{b}(R,  \nsubseteq Q; \Delta, \ba^t)$ by an application of Proposition \ref{PropositionUniformExistenceForTestElts} with $d = c$.
\end{definition}

\begin{remark}
It is also easy to see that
\[
\tau_{b}(R,  \nsubseteq Q; \Delta, \ba^t) = \sum_{n \geq k} \phi_{ne_0}(F^{ne_0}_*c \ba^{\lceil t(p^{ne_0} - 1) \rceil}).
\]
for any $k \geq 0$.
\end{remark}

Of course, one has to show that this definition is independent of the choice of $c$ and satisfies the condition of Definition \ref{DefnUltraGeneralizedTestIdeal}, but this is straightforward, see \cite[Section 6]{SchwedeFAdjunction} for details and compare with \cite{TakagiPLTAdjoint}.

\begin{remark}
 If $Q = (0)$, then this construction agrees with the usual construction of the test ideal, see Proposition \ref{propCharTestSmallestFixed}, Definition-Proposition \ref{DefnPropBigTestIdeals}, and \cite[Lemma 2.1]{HaraTakagiOnAGeneralizationOfTestIdeals}.
\end{remark}

\begin{remark}
It is relatively straightforward to see from Definition \ref{DefnAlternateAdjointTestIdeal}, that the formation of $\tau_{b}(R,  \nsubseteq Q; \Delta, \ba^t)$ commutes with localization with multiplicative sets $W$ that do not contain elements of $Q$.  This is because the $c$ from Proposition \ref{PropositionUniformExistenceForTestElts} still retains its defining property after localizing at such a multiplicative set.
\end{remark}

One substantial difference between this characterization of the big test ideal and the characterization used to prove previous discreteness and rationality results is that the exponent that appeared on $\ba$ previously was $\lceil t p^e \rceil$ instead of $\lceil t (p^e - 1) \rceil$.  Because we want to use the previous methods wherever we can, we will give a characterization of $\tau_{b}(R,  \nsubseteq Q; \Delta, \ba^t)$ using $\ba^{\lceil t p^{ne_0} \rceil}$.

\begin{lemma}
\label{LemmaAlternateUltraGeneralizedTestIdeal}
 Assume the notation of Definition \ref{DefnUltraGeneralizedTestIdeal}.  Then we have the following chain of containments:
\[
 \phi_{e_0}\left(F^{e_0}_* \ba^{\lceil t p^{e_0} \rceil} \tau_b(R, \nsubseteq Q; \Delta)\right) \subseteq \phi_{2 e_0}\left(F^{2e_0}_* \ba^{\lceil t p^{2 e_0} \rceil} \tau_b(R, \nsubseteq Q; \Delta)\right) \subseteq \ldots .
\]
Furthermore, the union of these ideals is equal to $\tau_{b}(R,  \nsubseteq Q; \Delta, \ba^t)$.
\end{lemma}
Before proving this lemma, let us remark that if $\Delta = 0$, $R = k[x_1, \ldots, x_n]$, and $Q = 0$, then this construction gives us the usual test ideal $\tau(R, \ba^{t})$.  This is particularly easy to see, since under these hypotheses $e_0 = 1$, $\tau_b(R, \not\subseteq Q; \Delta)=R$ and $\phi_{n e_0}(F^{ne_0}_* J) = J^{[1/p^{ne_0}]}$ as explained in Proposition \ref{propInverseFrobSplitting}. Therefore, compare the previous proposition with \cite[Definition 2.9, Proposition 2.22]{BlickleMustataSmithDiscretenessAndRationalityOfFThresholds}.

\begin{proof}
 We first prove the containments hold.  Note that

\begin{align*}
\phi_{n e_0}\left(F^{n e_0}_* \ba^{\lceil t p^{n e_0} \rceil} \tau_b(R, \nsubseteq Q; \Delta)\right) &= \phi_{n e_0}\left(F^{n e_0}_* \ba^{\lceil t p^{n e_0} \rceil} \phi_{e_0} (F^{e_0}_* \tau_b(R, \nsubseteq Q; \Delta))\right)\notag\\
&= \phi_{n e_0}\left(F^{n e_0}_* \phi_{e_0}(F^{e_0}_* (\ba^{\lceil t p^{n e_0} \rceil})^{[p^{e_0}]} \tau_b(R, \nsubseteq Q; \Delta))\right)\notag\\
&\subseteq \phi_{n e_0}\left(F^{n e_0}_* \phi_{e_0}(F^{e_0}_* \ba^{\lceil t p^{(n+1)e_0} \rceil} \tau_b(R, \nsubseteq Q; \Delta))\right)\notag\\
&= \phi_{(n+1) e_0}\left(F^{(n+1)e_0}_* \ba^{\lceil t p^{(n+1)e_0} \rceil} \tau_b(R, \nsubseteq Q; \Delta)\right)\notag.
\end{align*}

We now prove the statement about the union.  We first show that
\[
\sum_m \phi_{m e_0}(F^{m e_0}_* \ba^{\lceil t p^{m e_0} \rceil}  \tau_b(R, \nsubseteq Q; \Delta)) \supseteq \tau_{b}(R,  \nsubseteq Q; \Delta, \ba^t).
\]
Since $\ba$ does not vanish on $W$,  $\sum_m \phi_{m e_0}(F^{m e_0}_* \ba^{\lceil t p^{m e_0} \rceil}  \tau_b(R, \nsubseteq Q; \Delta))$ is not contained in $Q$.  However, for a fixed $n \geq 0$,
\[
\begin{split}
 \phi_{n e_0} ( F^{n e_0}_* \ba^{\lceil t(p^{n e_0} - 1) \rceil} &\sum_m \phi_{m e_0}(F^{m e_0}_* \ba^{\lceil t p^{m e_0} \rceil}  \tau_b(R, \nsubseteq Q; \Delta)) )\\
&=\sum_m \phi_{(n + m) e_0}(F^{(n + m) e_0}_* (\ba^{\lceil t(p^{n e_0} - 1) \rceil})^{[p^{m e_0}]} \ba^{\lceil t p^{m e_0} \rceil}  \tau_b(R, \nsubseteq Q; \Delta))\\
&\subseteq \sum_m \phi_{(n + m) e_0}(F^{(n + m) e_0}_* \ba^{\lceil t(p^{(n+m)e_0} - p^{m e_0}) \rceil} \ba^{\lceil t p^{m e_0} \rceil}  \tau_b(R, \nsubseteq Q; \Delta))\\
&\subseteq \sum_m \phi_{(n + m) e_0}(F^{(n + m) e_0}_* \ba^{\lceil t p^{(n+m) e_0} \rceil}  \tau_b(R, \nsubseteq Q; \Delta))\\
&\subseteq \sum_m \phi_{m e_0}(F^{m e_0}_* \ba^{\lceil t p^{m e_0} \rceil}  \tau_b(R, \nsubseteq Q; \Delta)).
\end{split}
\]

The minimality of $\tau_{b}(R,  \nsubseteq Q; \Delta, \ba^t)$ then guarantees that
\[
\sum_m \phi_{m e_0}(F^{m e_0}_* \ba^{\lceil t p^{m e_0} \rceil}  \tau_b(R, \nsubseteq Q; \Delta)) \supseteq \tau_{b}(R,  \nsubseteq Q; \Delta, \ba^t)
\]
as desired.

Now we prove the reverse inclusion.  First choose a $c \in R \setminus Q$ as in Proposition \ref{PropositionUniformExistenceForTestElts}.  Then we can find a $c' \in R \setminus Q$ such that $cc'$ also satisfies the condition of \ref{PropositionUniformExistenceForTestElts} and such that $c' \ba^{\lceil t(p^{ne_0} - 1) \rceil} \subseteq \ba^{\lceil t p^{ne_0} \rceil}$ for all $n \geq 0$.
Thus we see, for every $k \geq 0$, that
\begin{align}
\label{EqnInteriorContainment}
\tau_{b}(R,  \nsubseteq Q; \Delta, \ba^t) &= \sum_{n \geq k} \phi_{ne_0}(F^{ne_0}_* cc' \ba^{\lceil t(p^{ne_0} - 1) \rceil})\\
& \subseteq \sum_{n \geq k} \phi_{ne_0}(F^{ne_0}_* c \ba^{\lceil t p^{ne_0} \rceil})\notag\\
& \subseteq \sum_{n \geq k} \phi_{ne_0}(F^{ne_0}_* c \ba^{\lceil t(p^{ne_0} - 1) \rceil})\notag\\
& = \tau_{b}(R,  \nsubseteq Q; \Delta, \ba^t)\notag.
\end{align}

Note also that
\[
\tau_{b}(R,  \nsubseteq Q; \Delta) = \sum_{n \geq 0} \phi_{ne_0}(F^{ne_0}_* (cR) ),
\]
for the same choice of $c$ (above $(cR)$ is the ideal generated by $c$).  That is, if $c$ satisfies Proposition \ref{PropositionUniformExistenceForTestElts} for our $\ba^t$, then $c$ also satisfies Proposition \ref{PropositionUniformExistenceForTestElts} for $\ba = R$.  Choose a finite set of generators $x_{ik}$ for $\tau_{b}(R,  \nsubseteq Q; \Delta)$ such that $x_{ik} \in \phi_{ke_0}(F^{ke_0}_* (cR) )$ (note, $k$ will vary). Now

\begin{align}
\sum_{n \geq 0} \phi_{ne_0}(F^{ne_0}_* \ba^{\lceil t p^{n e_0} \rceil} x_{ik} ) &\subseteq \sum_{n \geq 0} \phi_{ne_0}(F^{ne_0}_* \ba^{\lceil t p^{n e_0} \rceil} \phi_{ke_0}( F^{ke_0}_* (cR) ))\notag\\
&\subseteq \sum_{n \geq 0} \phi_{ne_0}(F^{ne_0}_* \phi_{ke_0}( F^{ke_0}_* c \ba^{\lceil t p^{(n+k)e_0} \rceil}))\notag\\
&=\sum_{n \geq k} \phi_{ne_0}(F^{ne_0}_* c \ba^{\lceil t p^{ne_0} \rceil})\notag\\
&=\tau_{b}(R,  \nsubseteq Q; \Delta, \ba^t)\notag.
\end{align}
Where the last equality is due to (\ref{EqnInteriorContainment}).
But then we immediately see that
\[
\tau_{b}(R,  \nsubseteq Q; \Delta, \ba^t) \supseteq \sum_{n \geq 0} \phi_{ne_0}(F^{ne_0}_* \ba^{\lceil t p^{n e_0} \rceil} \tau_{b}(R,  \nsubseteq Q; \Delta))
\]
since the $x_{ik}$ generate $\tau_{b}(R,  \nsubseteq Q; \Delta)$.
\end{proof}

\begin{remark}
Experts in tight closure theory for pairs may feel more comfortable developing a tight closure theory with a different set of permissible test elements and defining $\tau_{b}(R, \nsubseteq Q; \Delta, \ba^t)$ as something like $\Ann_R(0^{*(\nsubseteq Q; \Delta, \ba^t)}_{E_R})$.  However, this is unnecessary and no easier (although possibly more familiar) than what we did above.  Another advantage of the above approach is that it does not require any local or Matlis duality.
\end{remark}

\subsection{Jumping numbers of generalized test ideals}

In this section, we define the $F$-jumping numbers and outline several results about test ideals which will be useful in both of our approaches to their discreteness and rationality.

We begin with an observation that the test ideal is constant if we increase the exponent very slightly.  This sort of behavior was first observed for test ideals in \cite[Remark 2.12]{MustataTakagiWatanabeFThresholdsAndBernsteinSato} and justifies the notion of jumping numbers.

\begin{lemma}
\label{LemmaPositiveIncreaseOk}
Suppose that $X = \Spec R$, $\Delta \geq 0$, $\ba \subseteq R$ and $Q \in \Spec R$ satisfy the conditions of Definition \ref{DefnUltraGeneralizedTestIdeal}.  Then for every positive real number $t > 0$, there exists an $\epsilon > 0$ such that
\[
 \tau_b(X, \nsubseteq Q; \Delta, \ba^t) = \tau_b(X, \nsubseteq Q; \Delta, \ba^{s})
\]
for all $s \in [t, t+\epsilon]$.  In particular, taking $Q = (0)$, the same statement holds for the test ideal $\tau_b(X; \Delta, \ba^t)$.
\end{lemma}
\begin{proof}
Set $e := {e_0}$ where $e_0$ is as in Definition \ref{DefnUltraGeneralizedTestIdeal}.  We have assumed that $\Hom_R(F^e_* R( (p^e - 1) \Delta), R)$ is free as an $F^e_* R$-module with $\phi_e : F^e_* R \rightarrow R$ corresponding to $\Delta$.
 Choose $c$ as in Proposition \ref{PropositionUniformExistenceForTestElts} and $b \in \ba \cap (R \setminus Q)$ (note that $cb$ also satisfies the condition of Proposition \ref{PropositionUniformExistenceForTestElts}).  Furthermore, there exists an integer $m > 0$ such that
\[
 \tau_b(X, \nsubseteq Q; \Delta, \ba^t) = \sum_{n = 0}^{m} \phi_{ne}(F^{ne}_* \ba^{\lceil t (p^{ne} - 1) \rceil} cb).
\]
Set $\epsilon = {1 \over p^{me} - 1}$.  Notice that we always have $\tau_b(X, \nsubseteq Q; \Delta, \ba^{t + \epsilon}) \subseteq \tau_b(X, \nsubseteq Q; \Delta, \ba^t)$, so we will prove the reverse inclusion.  Now note that we have

\begin{align}
\tau_b(X, \nsubseteq Q; \Delta, \ba^t) &= \sum_{n = 0}^{m} \phi_{ne}(F^{ne}_* \ba^{\lceil t (p^{ne} - 1) \rceil} cb)\notag\\
&\subseteq \sum_{n = 0}^{m} \phi_{ne}(F^{ne}_* \ba^{\lceil t (p^{ne} - 1) \rceil + 1} c)\notag\\
&\subseteq \sum_{n = 0}^{m} \phi_{ne}(F^{ne}_* \ba^{\lceil (t + \epsilon) (p^{ne} - 1) \rceil} c)\notag\\
&\subseteq \sum_{n = 0}^{\infinity} \phi_{ne}(F^{ne}_* \ba^{\lceil (t + \epsilon) (p^{ne} - 1) \rceil} c)\notag\\
&=\tau_b(X, \nsubseteq Q; \Delta, \ba^{t + \epsilon})\notag.
\end{align}
which completes the proof.
\end{proof}
In the spirit of Lemma \ref{LemmaPositiveIncreaseOk}, we define
\begin{definition}
\label{DefinitionJumpingNumber}
A positive real number $\alpha$ is called an \emph{$F$-jumping number of $(X, \nsubseteq Q; \Delta, \ba)$} if it is a \emph{jumping number} of the test ideals $\tau_b(X, \nsubseteq Q; \Delta, \ba^t)$ for $t>0$. That means
\[
\tau_b(X, \nsubseteq Q; \Delta, \ba^{\alpha})\neq \tau_b(X, \nsubseteq Q; \Delta, \ba^{\alpha-\epsilon})
\]
for all $\epsilon>0$.
\end{definition}

We begin with showing some basic properties of $F$-jumping numbers, the first of which is the fact that for appropriate $e>0$, the $(p^e)\th$ multiple of an $F$-jumping number is also an $F$-jumping number, cf. \cite[Proposition 3.4]{BlickleMustataSmithDiscretenessAndRationalityOfFThresholds}.

\begin{lemma}
\label{LemmaPAlphaJumps}
Suppose that $X = \Spec R$ is an $F$-finite normal scheme of characteristic $p > 0$ with a canonical module.  Further suppose that $\Delta$ is an effective $\bQ$-divisor, $\ba \subseteq R$ an ideal, $Q \in \Spec R$ and $e_0 > 0$ is a positive integer, all of which satisfy the conditions of Definition \ref{DefnUltraGeneralizedTestIdeal}. If $\alpha$ is a jumping number for $\tau_b(X, \nsubseteq Q; \Delta, \ba^t)$ then so is $p^{e_0} \alpha$.  In particular, if $Q = (0)$, then if $\alpha$ is a jumping number of $\tau_b(X; \Delta, \ba^t)$, then so is $p^{e_0} \alpha$.
\end{lemma}
\begin{proof}
Set $e := {e_0}$ where $e_0$ is as in Definition \ref{DefnUltraGeneralizedTestIdeal}.  The hypotheses of Definition \ref{DefnUltraGeneralizedTestIdeal} imply that $\Hom_R(F^e_* R( (p^e - 1) \Delta), R)$ is free as an $F^e_* R$-module and so we may choose $\phi_e : F^e_* \O_X \rightarrow \O_X$ corresponding to $\Delta$.
 As in Lemma \ref{LemmaAlternateUltraGeneralizedTestIdeal},
\[
\tau_b(X, \nsubseteq Q; \Delta, \ba^{\alpha}) =  \sum_{n \geq n_0} \phi_{ne}(F^{ne}_* \ba^{\lceil p^{ne} \alpha\rceil} \tau_b(X, \nsubseteq Q; \Delta))
\]
for any fixed $n_0 > 0$.
Therefore, if $p^e \alpha$ is not a jumping number, we have that
\[
 \tau_b(X, \nsubseteq Q; \Delta, \ba^{p^e \alpha}) = \tau_b(X, \nsubseteq Q; \Delta, \ba^{p^e \alpha - p^e \epsilon})
\]
for all sufficiently small $\epsilon > 0$.
This implies that
\[
 \sum_{n \geq 0} \phi_{ne}(F^{ne}_* \ba^{\lceil p^{(n+1)e} \alpha \rceil} \tau_b(X, \nsubseteq Q; \Delta)) = \sum_{n \geq 0} \phi_{ne}(F^{ne}_* \ba^{\lceil p^{(n+1)e } (\alpha - \epsilon)\rceil} \tau_b(X, \nsubseteq Q; \Delta))
\]
for all $\epsilon > 0$ sufficiently small.  Applying $F^e_*$ and $\phi_e$, we obtain that
\[
 \sum_{n \geq 0} \phi_{(n+1)e}(F^{(n+1)e}_* \ba^{\lceil p^{(n+1)e} \alpha \rceil } \tau_b(X, \nsubseteq Q; \Delta)) = \sum_{n \geq 0} \phi_{(n+1)e}(F^{(n+1)e}_* \ba^{\lceil p^{(n+1)e} (\alpha - \epsilon) \rceil} \tau_b(X, \nsubseteq Q; \Delta))
\]
for all $\epsilon > 0$ sufficiently small.  Therefore,
\[
\tau_b(X, \nsubseteq Q; \Delta, \ba^{\alpha}) = \tau_b(X, \nsubseteq Q; \Delta, \ba^{\alpha - \epsilon})
\]
for all sufficiently small $\epsilon > 0$.
But that is a contradiction.
\end{proof}

We also recall a Skoda-type theorem for test ideals.

\begin{lemma}
\label{LemmaSkodaTypeTheorem}
Suppose that $X = \Spec R$, $\Delta \geq 0$, $\ba \subseteq R$ and $Q \in \Spec R$ satisfy the conditions of Definition \ref{DefnUltraGeneralizedTestIdeal}.  Further suppose that $\ba$ is generated by $l$ elements.  Then for every $\alpha \geq l$, we have
\[
 \tau_b(X, \nsubseteq Q; \Delta, \ba^t) = \ba \tau_b(X, \nsubseteq Q; \Delta, \ba^{t-1}) .
\]
In particular, by assuming $Q = (0)$, the same statement holds for the test ideal $\tau_b(X; \Delta, \ba^t)$.
\end{lemma}
\begin{proof}
 The proof is essentially the same as \cite[Theorem 4.2]{HaraTakagiOnAGeneralizationOfTestIdeals}.  Set $e := {e_0}$ where $e_0$ is as in Definition \ref{DefnUltraGeneralizedTestIdeal}.  We assume that $\phi_e : F^e_* \O_X \rightarrow \O_X$ corresponds to $\Delta$.  We then note the following equality,
\[
 \ba^{\lceil p^{ne} t \rceil} = \ba^{\lceil p^{ne} (t - l) \rceil} \ba^{p^{ne} l}  = \ba^{\lceil p^{ne} (t - l) \rceil} \ba^{p^{ne} (l-1)}  \ba^{[p^{ne}]}  = \ba^{\lceil p^{ne}(t-1) \rceil} \ba^{[p^{ne}]}.
\]
Then for some $m > 0$ we have both

\begin{align*}
\tau_b(X, \nsubseteq Q; \Delta, \ba^t) &= \sum_{n = 1}^{m} \phi_{ne}(F^{ne}_* \ba^{\lceil t p^{ne} \rceil} \tau_b(X, \nsubseteq Q; \Delta)), \text{ and}\\
\tau_b(X, \nsubseteq Q; \Delta, \ba^{t-1}) &= \sum_{n = 1}^{m} \phi_{ne}(F^{ne}_* \ba^{\lceil (t-1) p^{ne} \rceil} \tau_b(X, \nsubseteq Q; \Delta)).
\end{align*}
However,
\[
\begin{split}
 \tau_b(X, \nsubseteq Q; \Delta, \ba^t) &= \sum_{n = 1}^{m} \phi_{ne}(F^{ne}_* \ba^{\lceil t p^{ne} \rceil} \tau_b(X, \nsubseteq Q; \Delta)) \\
 &= \sum_{n = 1}^{m} \phi_{ne}(F^{ne}_* \ba^{\lceil (t - 1) p^{ne} \rceil} \ba^{[p^{ne}]} \tau_b(X, \nsubseteq Q; \Delta)) \\
 &= \ba \sum_{n = 1}^{m} \phi_{ne}(F^{ne}_* \ba^{\lceil (t-1) p^{ne} \rceil} \tau_b(X, \nsubseteq Q; \Delta))\\ \
 &= \ba \tau_b(X, \nsubseteq Q; \Delta, \ba^{t-1})
 \end{split}
\]
as desired.
\end{proof}

\begin{corollary}
\label{CorOfSkodaTypeTheoremForJumping}
Suppose that $X = \Spec R$, $\Delta \geq 0$, $\ba \subseteq R$ and $Q \in \Spec R$ satisfy the conditions of Definition \ref{DefnUltraGeneralizedTestIdeal}.  Suppose that $\ba$ is generated by $l$ elements and that $\alpha > l$ is an $F$-jumping number for $\tau_b(X \nsubseteq Q; \Delta, \ba^{t})$, then so is $\alpha - 1$.
\end{corollary}

We end this section by noticing that, if we can prove discreteness and rationality of $F$-jumping numbers on an affine cover of a scheme, then discreteness and rationality also hold on the original scheme.

\begin{proposition}
\label{PropPassToAffineCover}
 Suppose that $(X, \Delta, \ba^t)$ is a triple and that $\{U_i = \Spec R_i\}$ is a finite affine cover of $X$.  Suppose that for each $i$, the $F$-jumping numbers of $\tau_b(U_i; \Delta|_{U_i}, \ba^t)$ are discrete and rational.  Then the $F$-jumping numbers of $\tau_b(X; \Delta, \ba^t)$ are also discrete and rational.
\end{proposition}
\begin{proof}
 Suppose that $t$ is an $F$-jumping number for $(X, \Delta, \ba^t)$.  That is, $\tau_b(X; \Delta, \ba^t) \neq \tau_b(X; \Delta, \ba^{t - \epsilon})$ for any $\epsilon > 0$.  This immediately implies that $t$ is also an $F$-jumping number on one of the $U_i$.  In particular, $t$ is a rational number.  On the other hand, if $\{t_j \}$ is a collection of $F$-jumping numbers with an accumulation point $t$, then $t$ is also an accumulation point of $F$-jumping numbers on one of the $U_i$ (since there are only finitely many $U_i$), which is impossible by assumption.
\end{proof}

\section{Discreteness and rationality for {$k$}-algebras}
\label{SectionDiscRatKalg}

In this section we show that the set of $F$-jumping numbers of a triple $(R, \Delta, \ba)$ is a discrete set of rational numbers in the case that \emph{$R$ is essentially of finite type over an $F$-finite field} $k$ and $K_R + \Delta$ is $\bQ$-Cartier with index not divisible by $p > 0$. We use the methods described in Section \ref{SectionGenTest} -- in particular Theorem \ref{ThmRestrictionForTestIdeals} -- to reduce to the case $R=k[x_1,\ldots,x_n]$ and then to prove discreteness and rationality for the variant of the test ideal introduced in Section \ref{SectionRestrictionOfTestIdeals}. In the polynomial ring case we can then employ the same strategy as in \cite{BlickleMustataSmithDiscretenessAndRationalityOfFThresholds} by universally bounding the degrees of the generators of our variant of the test ideal in question.

\begin{theorem}
\label{ThmMainNonPrincipalCase}
Suppose that $(X, \Delta, \ba)$ is a triple where $X$ is normal and \emph{essentially of finite type over an $F$-finite field} of characteristic $p > 0$, and $K_X + \Delta$ is $\bQ$-Cartier with index not divisible by $p$. Then the set of jumping numbers of $\tau_b(X; \Delta, \ba^t)$ is a discrete set of rational numbers.
\end{theorem}

The proof of this theorem will follow after a series of reduction steps have been carried out. In the course of this we may and will assume -- justified by Proposition \ref{PropPassToAffineCover} --  that $\O_X((p^e - 1)(K_X + \Delta))$ is free of rank one.  The following lemma shows that one can reduce to the finite type case, the one afterwards covers the crucial reduction to the polynomial ring case.

\begin{lemma}
\label{LemFiniteTypeImpliesEssentFiniteType}
Suppose that the set of $F$-jumping numbers for triples $(R; \Delta, \ba)$ is discrete and rational in the case that $R$ is normal and of \emph{finite type} over an $F$-finite field $k$, and $K_R+\Delta$ is $\bQ$-Cartier with index not divisible by $p > 0$.  Then the set of $F$-jumping numbers is discrete and rational for triples $(R; \Delta, \ba^t)$ with $R$ only \emph{essentially of finite type} over $k$ (but otherwise satisfying the same assumptions as above).
\end{lemma}
\begin{proof}
Suppose that $R = S^{-1} A$ where $A$ is of finite type over an $F$-finite field and $S$ is a multiplicative set.  For each point $Q \in X = \Spec R \subseteq \Spec A = Y$, there exists an element $f_Q' \in A$ such that $U_Q' = \Spec A_{f_Q'} \subseteq Y$ is an open set containing $Q$ and that $U_Q'$ is normal.  We set $V_Q' = \Spec R_{f_Q'} \subseteq X$ and note that $V_Q' = U_Q' \cap X$.  We may extend $\Delta|_{V_Q'}$ to $\Delta_{U_Q'}$ on $U_Q'$ for each $Q$ by pull-back.  Choosing appropriate $f_Q \in A$, we may find $U_Q = \Spec A_{f_Q} \subseteq \Spec A_{f_Q'} = U_Q'$ so that $(K_{U_Q} + \Delta_{U_Q'}|_{U_Q})$ is $\bQ$-Cartier with index not divisible by $p > 0$.  Set $V_Q$ to be $U_Q \cap X = \Spec R_{f_Q}$.

Note that we do have discreteness and rationality on each of the $U_Q$ by hypothesis and thus also on each of the $V_Q$.  But we can cover $X$ by finitely many of the $U_Q$ and so we are done by Proposition \ref{PropPassToAffineCover}.
\end{proof}

\begin{lemma}
\label{LemAppliedFedderDiscretenessAndRationality}
Let $S = k[x_1, \ldots, x_m]$ be a polynomial ring over an $F$-finite field $k$, $Q$ be a prime ideal of $S$ and $\mathfrak{b}$ be an ideal of $S$ such that $\mathfrak{b} \nsubseteq Q$.
Set $R=S/Q$ and $\ba=\mathfrak{b}R$.
Suppose that $(X = \Spec R; \Delta_R, \ba)$ is a triple such that $\O_X((1-p^e)(K_X + \Delta_R))$ is free.  Set $\Delta_S$ to be a $\bQ$-divisor on $Y:=\Spec S$ related to $\Delta_R$ by $F$-adjunction as in Theorem \ref{TheoremFedderFadjunction}.
Then the discreteness and rationality of the jumping numbers of $\tau_b(Y, \nsubseteq Q; \Delta_S, \mathfrak{b}^t)$ implies the discreteness and rationality of the jumping numbers for $\tau_b(X; \Delta_R, \ba^t)$.
\end{lemma}
\begin{proof}
 This follows immediately since for all $t \geq 0$ we have \[\tau_b(Y, \nsubseteq Q; \Delta_S, \mathfrak{b}^t) S/Q = \tau_b(X; \Delta_R, \ba^t)\] by Theorem \ref{ThmRestrictionForTestIdeals}.
\end{proof}

The preceding two lemmata show that it is sufficient to prove discreteness and rationality of the jumping numbers for the variant of the test ideal $\tau_b(Y, \nsubseteq Q; \Delta_S, \mathfrak{b}^t)$ in the case that $Y=\Spec k[x_1,\ldots,x_m]$ and $k$ is an $F$-finite field. Hence, the proof of Theorem \ref{ThmMainNonPrincipalCase} is completed once the following proposition has been proven.

\begin{proposition}
\label{PropDiscRatUltraGenAffineSpace}
Let $S=k[x_1,\ldots,x_m]$ be a polynomial ring over an $F$-finite field $k$. and suppose that $(Y=\Spec S,\nsubseteq Q; \Delta_S, \ba)$ otherwise satisfies the conditions in Definition \ref{DefnUltraGeneralizedTestIdeal}. Then the set of jumping numbers of $\tau_b(Y,\nsubseteq Q; \Delta, \ba^t)$ is a discrete set of rational numbers.
\end{proposition}

\begin{proof}
The proof is essentially the same as \cite[Proof of Theorem 3.1]{BlickleMustataSmithDiscretenessAndRationalityOfFThresholds}. The key ingredient in op.~cit.~is a degree bound for the generators of the test ideal, which we provide in Proposition \ref{PropBoundOnDegreeForTestIdealVariant} below for the variant $\tau_b(Y,\nsubseteq Q; \Delta, \ba^t)$. The argument proceeds as follows.

Suppose that $\ba$ is generated by polynomials of degree $\leq a$. If $\alpha$ is an accumulation point of jumping numbers, we may -- by Lemma \ref{LemmaPositiveIncreaseOk} -- choose an \emph{increasing} sequence $\alpha_i$ of jumping numbers converging to $\alpha$. Since $\alpha_i \leq \alpha$ it follows by Proposition \ref{PropBoundOnDegreeForTestIdealVariant} below that \emph{for each $i$}, the test ideal
\[
    \tau_b(Y,\nsubseteq Q; \Delta, \ba^{\alpha_i})
\]
is generated by polynomials of degree less than or equal to $\lfloor a \cdot \alpha + C \rfloor$ where $C$ is a constant only depending on $Y$, $Q$ and $\Delta$, but not on $\ba$ or $t$. This implies that the descending sequence of test ideals $\tau_b(Y,\nsubseteq Q; \Delta, \ba^{\alpha_i})$ is determined by a descending sequence of subspaces in the finite dimensional $k$-vector space of polynomials of degree less or equal to $\lfloor a \cdot \alpha + C \rfloor$. Hence this descending sequence must stabilize, contradicting the assumption that each $\alpha_i$ is a jumping number. This shows the discreteness, and the rationality follows with the standard argument using the $p^e$-periodicity (Lemma \ref{LemmaPAlphaJumps}) and Skoda's theorem (Corollary \ref{CorOfSkodaTypeTheoremForJumping}) just as in \cite[Proof of Theorem 3.1]{BlickleMustataSmithDiscretenessAndRationalityOfFThresholds}.
\end{proof}

It remains to show the degree bound for the test ideals $\tau_b(Y,\nsubseteq Q; \Delta, \ba^t)$. We start with the following simple lemma, cf.~\cite{AndersonElementaryLFunctions}.

\begin{lemma}
\label{LemmaDegreeBoundCartierLinea}
Let $S=k[x_1,\ldots,x_m]$ be the polynomial ring over the $F$-finite field $k$, and let $\phi: S \to S$ be a $p^{-e}$-linear map. Then there is a constant $\delta$ such that for all $r \in S$ and all $n > 0$ one has
\[
\deg (\phi^n(r)) \leq \frac{\deg(r)}{p^{ne}}+\frac{\delta}{p^e-1}.
\]
\end{lemma}
\begin{proof}
Let $\{b_j\}_{1 \leq j \leq l}$ be a $k^{p^e}$-basis of $k$. We know that $S$ is freely generated by the monomials $\{b_jx^{\underline{i}}=b_jx_1^{i_1}x_2^{i_2}\cdot\ldots\cdot x_m^{i_m}\,|\, 0\leq i_j \leq p^e-1 , 1 \leq j \leq l\, \}$ as an $S^{p^e}$-module. Let $\delta = p^e \cdot \max_{\underline{i},j} \{ \deg(\phi(b_jx^{\underline{i}})) \}$. Then, writing $r = \sum r_{\underline{i},j}^{p^e}b_jx^{\underline{i}}$ uniquely for $r_{\underline{i},j} \in S$, we see that $p^e\deg(r_{\underline{i},j})\leq \deg(r)$ for each multi-index $\underline{i}$ (this requires a small (omitted) argument that uses the monomial shape of the basis). Since $\phi(r)=\sum r_{\underline{i},j}\phi(b_jx^{\underline{i}})$ by $p^{-e}$-linearity we get
\[
    \deg(\phi(r)) \leq \max_{\underline{i},j} \{\deg{r_{\underline{i},j}}\} + \frac{\delta}{p^e} \leq \frac{\deg(r)}{p^e}+\frac{\delta}{p^e}\left(\leq \frac{\deg(r)}{p^e}+\frac{\delta}{p^e-1}\right).
\]
We now proceed by induction:
Assume
\[
 \deg(\phi^n(r)) \leq \frac{\deg(r)}{p^{ne}}+\frac{\delta\cdot(1 + \dots + p^{(n-1)e}) }{p^{ne}}
\]
for some $n \geq 1$.
Then
\[
\begin{split}
\deg(\phi^{n+1}(r)) &= \deg(\phi(\phi^n(r)))\\
                    &\leq \left( \frac{\deg(r)}{p^{ne}}+\frac{\delta\cdot(1 + \dots + p^{(n-1)e}) }{p^{ne}} \right) \Big/ p^e + {\delta \over p^e} \\
                    &= \frac{\deg(r)}{p^{(n+1)e}}+\frac{\delta\cdot(1 + \dots + p^{ne}) }{p^{(n+1)e}} \leq \frac{\deg(r)}{p^{(n+1)e}} + {\delta \over p^e - 1}.
\end{split}
\]
\end{proof}

\begin{proposition}[{cf.~\cite[Proposition 3.2]{BlickleMustataSmithDiscretenessAndRationalityOfFThresholds}}]
\label{PropBoundOnDegreeForTestIdealVariant}
Let $S=k[x_1,\ldots,x_m]$ be a polynomial ring over an $F$-finite field $k$ and suppose in addition that $(Y=\Spec S,\nsubseteq Q; \Delta_S, \ba)$ satisfies the conditions in Definition \ref{DefnUltraGeneralizedTestIdeal}. Assume that
\begin{itemize}
\item $\ba$ is generated in degree $\leq a$, and
\item let $\delta_e=\delta$ be a bound satisfying the conditions of the preceding Lemma \ref{LemmaDegreeBoundCartierLinea} for $\phi_e$, a map $F^e_*S \to S$ corresponding to $\Delta_S$ via (\ref{EqnGlobalBijection}) on page \pageref{EqnGlobalBijection}.
\end{itemize}
Then the test ideal $\tau_b(Y, \nsubseteq Q; \Delta_S, \fra^t)$ is generated in degree $\leq m+\lfloor ta + \frac{\delta}{p^e-1} \rfloor$
\end{proposition}
\begin{proof}
Suppose that $\tau_b(Y, \nsubseteq Q; \Delta_S)$ can be generated by polynomials of degree at most $d$. By Lemma \ref{LemmaAlternateUltraGeneralizedTestIdeal}, we may choose some integer $n$ such that
\[
 \tau_b(Y, \nsubseteq Q;  \Delta_S, \fra^t) = \phi_{ne}(F^{ne}_* \fra^{\lceil t p^{ne} \rceil} \tau_b(Y, \nsubseteq Q;  \Delta_S)).
\]
Making $n$ larger is harmless, so we may assume that $\frac{a+d}{p^{ne}} < 1-\{\frac{\delta}{p^e-1}+ta\}$. Since $\fra^{\lceil t p^{ne} \rceil}  \tau_b(Y, \nsubseteq Q;  \Delta_S)$ is generated by elements of degree at most $\lceil tp^{ne} \rceil a + d$, the $S$-module $F^{ne}_*\fra^{\lceil t p^{ne} \rceil}  \tau_b(Y, \nsubseteq Q;  \Delta_S)$ can be generated by elements of degree $\leq (p^{ne}-1)m + \lceil tp^{ne} \rceil a + d$. This follows simply since -- as just observed in the proof of Lemma \ref{LemmaDegreeBoundCartierLinea} -- the $S$-module $F^{ne}_*S$ is generated by elements of degree $\leq (p^{ne}-1)m$. Hence by Lemma \ref{LemmaDegreeBoundCartierLinea}, $\tau_b(Y, \nsubseteq Q;  \Delta_S, \fra^t)$ can be generated in degree at most
\[
{(p^{ne}-1)m+\lceil tp^{ne} \rceil a + d \over p^{ne} } + \frac{\delta}{p^e-1}\leq m+ta+\frac{a+d}{p^{ne}}+\frac{\delta}{p^e-1}.
\]
Since the degree is an integer we may take the round down, and since we arranged that the term $\frac{a+d}{p^{ne}} < 1-\{\frac{\delta}{p^e-1}+ta\}$, this term does not affect the round down. Hence we get the upper bound $m+ \lfloor ta+\frac{\delta}{p^e-1} \rfloor$ as claimed.
\end{proof}

\begin{remark}
The bound we obtained in the preceding proposition is not optimal. In fact, the appearance of the term $m$ can be avoided with the help of a slightly refined argument. Since for our purpose only the universality of the degree bound is essential, we only sketch the proof for the better bound in the case that $k$ is perfect:

As before in Remark \ref{remarkCartierMap} we take the basis $\{e_{e,{\underline{i}}}={\underline{x}}^{\underline{i}}\}_{0 \leq i_j \leq p^e-1}$ of $F^e_*S$ over $S$, and recall that the Cartier map
\[
    \psi_e({\underline{x}}^i)=\begin{cases} {\underline{x}}^{((i+\underline{1})/p^e)-\underline{1}} & \text{if $p^e$ divides each $i_j+1$} \\
    0 & \text{otherwise},
    \end{cases}
\]
where $\underline{1}=(1,\ldots,1)$. We saw in Remark \ref{remarkCartierMap} that the maps $\psi_{e,{\underline{i}}}=\psi_e(F^e_*(\underline{x}^{p^e-1}/\underline{x}^{\underline{i}})\cdot\blank)$ form an $R$-module basis of $\Hom_R(F^e_*R,R)$. Each of the $\psi_{e,{\underline{i}}}$ has the property that it sends a basis element $e_{e,{\underline{i}}'}$ to 1 if and only if ${\underline{i}}'={\underline{i}}$, and to zero otherwise (it is the basis dual to $x^{\underline{i}}$. In particular, the proof of Lemma \ref{LemmaDegreeBoundCartierLinea} hence shows that $\deg(\psi_{e,{\underline{i}}}(r)) \leq \frac{\deg(r)}{p^e}$ for each $\psi_{e,{\underline{i}}}$.

Also in Remark \ref{remarkCartierMap} we observed that $\psi_e$ is an $F^e_*R$-module generator of $\Hom(F^e_*R,R)$. Hence, an arbitrary $\phi_e \in \Hom(F^e_*R,R)$ may be written as $\phi_e(\blank)=\psi_e(F^e_*f_e \cdot \blank)$ for some $f_e \in S$. If an ideal $J$ is generated by $h_1,\ldots,h_k$, then
\[
    \phi_{e}(F^e_*J)=\psi_{e}(F^e_*f_e\sum_j F^e_*S h_j) = \sum_{j} \psi_e(\sum_{{\underline{i}}}SF^e_*{\underline{x}}^{\underline{i}} f_e h_j)=\sum_{j,{\underline{i}}}S\psi_{e,{\underline{i}}}(f_e h_j)
\]
and hence the ideal $\phi_{e}(F^e_*J)$ is generated by the elements $\psi_{e,{\underline{i}}}(f_e h_j)$. If each $\deg h_j \leq b$, then $\phi_{e}(F^e_*J)$ is generated in degree $\leq \frac{b}{p^e}+\frac{\deg f_e}{p^e}$ by the above observation about the $\psi_{e,{\underline{i}}}$. Recall that, if $\phi_e=\psi_e(F^e_*f_e\cdot \blank)$, then $\phi_{ne}=\psi_{ne}(F^e_*f^{1+p^e+\ldots+p^{(n-1)e}}\cdot \blank)$. Hence, applying the above for $e$ replaced with $ne$, we get that $\phi_{ne}(J)$ is generated in degree $\leq \frac{b}{p^{ne}}+\frac{\deg f_e(1+p^e+\ldots+p^{(n-1)e}}{p^{ne}}\leq \frac{b}{p^{ne}}+\frac{\deg f_e}{p^e-1}$. Using this bound in the proof of the preceding Proposition one obtains the improved bound $\lfloor ta + \frac{\deg f_e}{p^e-1} \rfloor$ for the generators of the test ideal $\tau_b(S,\nsubseteq Q;\Delta,\ba^t)$.
\end{remark}



\begin{remark}
To make this statement even more effective, one would need to control the degree of $f_e$.
In general, we don't know how to this, however, if the coordinate ring of $X$ is a quasi-Gorenstein graded ring and $\Delta_R=0$ in the situation of Lemma \ref{LemAppliedFedderDiscretenessAndRationality}, then we can describe the degree of $f_e$ very explicitly.

Let $S:=k[x_1, \dots, x_m]$ be a polynomial ring over an $F$-finite field $k$, and assume that $R:=S/Q$ is a quasi-Gorenstein normal graded ring with $R_0=k$.
First note that $f_e$ is a generator for the cyclic graded $R$-module $(Q^{[q]}:Q)/Q^{[q]}$ for $q=p^e$.
By \cite[Lemma 1.6]{FedderFPureRat} (which is stated only in the case of local rings, but also works in the case of graded rings), there exists a degree-preserving isomorphism
$$(Q^{[q]}:Q) {}^*\!\Hom_{S^q}(S, S^q)/ Q^{[q]} {}^*\!\Hom_{S^q}(S, S^q) \cong {}^*\!\Hom_{R^q}(R, R^q)$$
for all $q=p^e$.
Let $a(R), a(S)$ be the $a$-invariants of $R$ and $S$, respectively (see \cite{GotoWatanabeOnGradedRings} for the definition of $a$-invariants).
Since $\omega_R \cong R(a(R))$ by assumption, $\omega(R^q) \cong R^q(a(R^q)) \cong R^q(q \cdot a(R))$.
Then
\begin{align*}
{}^*\!\Hom_{R^q}(R, R^q) \cong {}^*\!\Hom_{R^q}(R, \omega_{R^q})(-q \cdot a(R)) & \cong \omega_R(-q  \cdot a(R))\\
& \cong R((1-q) a(R)).
\end{align*}
Similarly, one has ${}^*\!\Hom_{S^q}(S, S^q) \cong S((1-q)a(S))$.
Thus, $(Q^{[q]}:Q)/Q^{[q]}$ is a cyclic graded $R$-module generated by an element of degree $(q-1)(a(R)-a(S))$.
That is, $f_e$ has degree $(q-1)(a(R)-a(S))$. Thus, if the ideal $\mathfrak{b} \subseteq S$ is generated by polynomials of degree at most $b$, then $ \tau_b(Y, \nsubseteq Q;  \Delta_S, \mathfrak{b}^t)$ can be generated by polynomials of degree at most $a(R)-a(S)+\lfloor tb \rfloor$.
\end{remark}






\section{Discreteness and rationality for principal ideals}
\label{SectionDiscretenessAndRationalityForPrincipal}

The remainder of this paper is devoted to proving discreteness and rationality of $F$-jumping numbers for triples $(X, \Delta, \ba^t)$ where $X$ is any normal scheme having a canonical divisor $K_X$ such that $K_X + \Delta$ is $\bQ$-Cartier with index not divisible by $p$, and $\ba$ is principal. In the case that $X$ is regular and $F$-finite this result has been obtained in \cite{BlickleMustataSmithFThresholdsOfHypersurfaces} using $D$-module techniques and in \cite{Blickle.MinimalGamma} using variants of a fundamental theorem of Hartshorne and Speiser \cite[Proposition 1.11]{HartshorneSpeiserLocalCohomologyInCharacteristicP} on nilpotence in co-finite modules with a Frobenius action. The same result of op.~cit.~ has been used directly in \cite{KatzmanLyubeznikZhangOnDiscretenessAndRationality} to show rationality and discreteness in the case of an excellent regular local ring, however without assuming $F$-finiteness. This central result of Hartshorne and Speiser roughly says that for a local ring $R$ and a co-finite $R$-module $M$ with a left Frobenius action $F_M$, the \emph{ascending chain} of kernels $\ker F^e_M$ stabilizes after finitely many steps.

\subsection{Stabilization of images of \texorpdfstring{$p^{-e}$}{Cartier}-linear maps}

The key result in our treatment of test ideals for principal ideals is the following globalized dual version of the Hartshorne-Speiser result. In the general form stated below, the result can be found in \cite[Lemma 13.1]{Gabber.tStruc}. This result is also interesting in other related contexts, see for example \cite{BlickleBoeckle.CartierFiniteness,Blickle.MinimalGamma}.
\begin{theorem}[\cite{Gabber.tStruc} Lemma 13.1]
\label{TheoremDualHSLForFFinite}
Let $X$ be a locally noetherian scheme over $\mathbb{F}_{p^e}$. Let $M$ be a coherent $\O_X$-module equipped with a $p^{-e}$-linear map $\phi: M \to M$. Then the descending   sequence of images
\[
    M \supseteq \phi(M) \supseteq \phi^2(M) \supseteq \phi^3(M) \ldots
\]
eventually stabilizes.
\end{theorem}
\label{ThmCartierImagesStabilize}
\begin{proof}
We can and will assume that $X=\Spec R$ is affine. Note that $\phi(M)$ is an $R$-submodule of $M$ since $r\phi(m)=\phi(r^{p^e} m)$. The action of $\phi$ on $M$ clearly restricts to an action on $\phi(M)$. Furthermore, if $S$ is any multiplicative subset of $R$, then $\phi(S^{-1}M)=\phi(S^{-p^e}M)=S^{-1}\phi(M)$, which shows that the formation of the image of the map $\phi$ commutes with localization. This, combined with the functoriality of the image implies that
\[
    Y_n \defeq \supp(\phi^n(M)/\phi(\phi^{n}(M)))=\{x \in X \colon \phi_x|_{\phi_x^n(M_x)} \text{ is not surjective }\}
\]
is a \emph{descending} sequence of closed subsets of $X = \Spec R$.  Since $X$ is noetherian this sequence must stabilize.  In other words, there exists $n\geq 0$ such that for all $m \geq n$ we have $Y_n=Y_m(\defeq Y)$. By replacing $M$ by $\phi^n(M)$ we may hence assume that for all $n$ we have $Y=\supp(\phi^n(M)/\phi(\phi^{n}(M)))$.

The statement that the chain $\phi^n(M)$ stabilizes means precisely that $Y$ is empty. So let us assume otherwise and let $Z=V(\frp)$ be an irreducible component of $Y$. Localizing at $\frp$ (and replacing $R$ by $R_{\frp}$, $M$ by $M_{\frp}$) we may hence assume that $X=\Spec R$, where $(R,\frp)$ is a local ring and $\supp(\phi^n(M)/\phi^{n+1}(M))=\{\frp\}$ for all $n$. In particular, for $n=0$, we get that there is $k > 0$ such that $\frp^k M \subseteq \phi(M)$. Then for $x \in \frp^k$,
\[
    x^2M \subseteq x\frp^k M \subseteq x \phi(M) = \phi(x^{p^e}M) \subseteq \phi(x^2M)
\]
and by iterating we obtain, for all $x\in \frp^k$ and all $a \in \mathbb{N}$, that $x^2M \subseteq \phi^a(M)$. Hence
\[
    \frp^{k(b+1)}M \subseteq (\frp^k)^{[2]}M \subseteq \phi^a(M) \text{ for all } a \in \mathbb{N}
\]
where $b$ is the number of generators of $\frp^k$ and $(\frp^k)^{[2]}$ is the ideal generated by the squares of the elements of $\frp^k$ (and it is easy to check that there is an inclusion $\frp^{k(b+1)} \subseteq (\frp^k)^{[2]}$). Hence the chain $\phi^a(M)$ stabilizes if and only if the chain $\phi^a(M)/\frp^{k(b+1)}M$ does. But this is a chain in $M/\frp^{k(b+1)}M$, which has finite length. This contradicts our assumption that $\supp(\phi^n(M)/\phi^{n+1}(M)) \neq \emptyset$ for all $n$.
\end{proof}
\begin{remark} In the case that $(R, \bm)$ is $F$-finite and local, Theorem \ref{ThmCartierImagesStabilize} is via Matlis duality equivalent to the above mentioned result of Hartshorne and Speiser saying that for a $p^e$-linear map $\psi : N \rightarrow N$ on an Artinian module $N$, the kernels of $\psi^n$ stabilize as $n$ increases.  To see this, it is sufficient to observe that a $p^{-e}$-linear map corresponding to $\phi : F^e_* M \rightarrow M$ is Matlis dual to a $p^e$-linear map corresponding to $\psi : M^{\vee} \rightarrow F^e_* (M^{\vee})$ (where $M^{\vee}$ is the Matlis dual of $M$).

If we apply $\Hom_R(\blank, E_{R})$ to $\phi$, we obtain a map
\[
 \Hom_R(M, E_{R}) \to[\phi^{\vee}] \Hom_R(F^e_* M, E_{R}),
\]
where $E_R$ is the injective hull of $R/\bm$.  But one also has
\[
\begin{split}
\Hom_R(F^e_* M, E_{R}) &\cong \Hom_R(F^e_* M \tensor_{F^e_* R} F^e_* R, E_{R}) \cong \Hom_{F^e_* R}(F^e_* M, \Hom_R(F^e_* R, E_{R})) \\
&\cong \Hom_{F^e_* R}(F^e_* M, E_{F^e_* R}) \cong F^e_* \Hom_R(M, E_{R}).
\end{split}
\]
This is because, as a commutative ring on its own, $F^e_*R$ has the same ring structure as $R$, and hence $E_{F^e_*R}$ is the same as $E_R$ as abelian groups; but the $R$-module structure of $E_{F^e_*R}$ is via $F^e$, consequently $E_{F^e_*R}\cong F^e_*E_R$ as $R$-modules.  Note that some of the indentifications used are not canonical, but they are unique up to multiplication by units (which does not impact the kernels of the maps).
\end{remark}
\subsection{Test ideals of principal ideals}
With this result at hand we can proceed to show discreteness and rationality for triples $(R, \Delta, f^t)$ where $R$ is $F$-finite and normal and $K_R + \Delta$ is $\bQ$-Cartier with index not divisible by $p$.

The key step is the following Proposition which says that rational numbers of the form $a/(p^{me}-1)$ cannot be accumulation points of jumping numbers (cf. \cite[Remark 2.12]{BlickleMustataSmithFThresholdsOfHypersurfaces}, \cite[Proposition 6.3]{KatzmanLyubeznikZhangOnDiscretenessAndRationality},  and \cite[Theorem 7.12]{SchwedeTakagiRationalPairs}).

\begin{proposition}
\label{PropLessThanRatIsOk}
Suppose that $R$ is an $F$-finite normal ring of characteristic $p > 0$ with a canonical module, and suppose that $\Delta$ is an effective $\bQ$-divisor such that $(p^e - 1)(K_R + \Delta)$ is Cartier  for some integer $e > 0$. Let $f \in R^{\circ}$. Then for every rational number of the form  $\frac{a}{p^{me} - 1}$ with $a,m \in \mathbb{N}$, there exists a positive real number $c < \frac{a}{p^{me} - 1}$ such that $\tau_b(R; \Delta, f^t)$ is constant for $t \in [c, \frac{a}{p^{me} - 1})$.
\end{proposition}

Before we can prove this, we need a description of the test ideal $\tau_b(R;\Delta,f^{\frac{a}{p^m}})$ as an image of an appropriate $p^{-ne}$-linear map.
\begin{lemma}
\label{LemTestIdealAsImage}
In the context of Proposition \ref{PropLessThanRatIsOk}, assume furthermore that $\Hom_R(F^e_* R((p^e - 1)\Delta), R)$ is free as an $F^e_* R$-module. Suppose that $\phi_e : F^e_* R \rightarrow R$ corresponds to $\Delta$ via Equation (\ref{EqnGlobalBijection}) on page \pageref{EqnGlobalBijection}. Then
\[
    \phi_{ne}(F^{ne}_* f^{p^{ne} a/p^m}\tau_b(R;\Delta) ) = \tau_b(R;\Delta, f^{a/p^m})
\]
for any $ne \geq m$.
\end{lemma}
\begin{proof}
 First note that for any $n_0$,
\[
\begin{split}
\tau_b(R;\Delta, f^{a/p^m}) = \sum_{n, ne \geq m} \phi_{ne}(F^{ne}_* f^{p^{ne} a/p^m} \tau_b(R;\Delta)) = \\
\sum_{n, ne \geq m} \phi_{ne}(F^{ne}_* f^{a p^{ne - m} } \tau_b(R;\Delta)) = \sum_{n \geq n_0, ne \geq m} \phi_{ne}(F^{ne}_* f^{a p^{ne - m} } \tau_b(R;\Delta)).
\end{split}
\]
But then, for any $n$ such that $ne > m$ we observe that
\[
\begin{split}
 \phi_{(n+1)e} \left(F^{(n+1)e}_* f^{a p^{(n+1)e - m} } \tau_b(R;\Delta)\right) =
\phi_{ne}\left(F^{ne}_* \phi_e\left(F^{e}_*\left( f^{a p^{ne - m} } \right)^{p^e} \tau_b(R;\Delta) \right) \right)=
\\
=\phi_{ne}\left(F^{ne}_* f^{ap^{ne - m} } \phi_e\left(F^{e}_*(\tau_b(R;\Delta)) \right) \right) = \phi_{ne}\left(F^{ne}_* f^{ap^{ne - m} } \tau_b(R;\Delta) \right).
\end{split}
\]
since $\phi_e\left(F^{e}_*(\tau_b(R, \Delta)) \right) = \tau_b(R, \Delta)$.  But this means that the sum stabilizes at the first stage, as desired.
\end{proof}

Now we are in a position to prove Proposition \ref{PropLessThanRatIsOk}.

\begin{proof}[Proof of Proposition \ref{PropLessThanRatIsOk}]
It is sufficient to check the statement at a finite number of affine charts of $\Spec R$.  Thus we may assume that $\Hom_R(F^e_* R( (p^e - 1) \Delta), R)$ is free so that $\phi_e : F^e_* R \rightarrow R$ corresponds to $\Delta$.
Without loss of generality we may also assume that $m = 1$ where $m$ is from the statement of Proposition \ref{PropLessThanRatIsOk}, since if not, simply replace $e$ by $me$.

Now consider the map $\psi_e : F^e_* \tau_b(R;\Delta) \rightarrow \tau_b(R;\Delta)$ defined by $\psi_e(\blank) = \phi_e (f^{a} \times \blank )$. The image of $\psi_e$ is contained in $\tau_b(R;\Delta)$ because the image of $\phi_e$ is. Note also that in our convention from Remark \ref{RemarkDefinitionOfComposedPInverseLinear} we have $\psi_{ne}(\blank)=\phi_{ne}(f^{a(1 + p^e + p^{2e} + \dots + p^{(n-1)e}) } \times \blank)$. Now, the preceding lemma implies that
\[
\begin{split}
\psi_{ne}(F^{ne}_* \tau_b(R;\Delta)) = \phi_{ne}(F^{ne}_* f^{a(1 + p^e + p^{2e} + \dots + p^{(n-1)e}) }\tau_b(R;\Delta)) = \\
= \tau_b(R;\Delta, f^{a (1 + p^e + p^{2e} + \dots + p^{(n-1)e}) \over p^{ne} }) = \tau_b(R;\Delta, f^{{a \over p^{ne}}{p^{ne} - 1 \over p^e - 1}}).
\end{split}
\]
By Theorem \ref{ThmCartierImagesStabilize} there exists $N$ sufficiently large such that
\[
 \psi_{Ne}(F^{Ne}_* \tau_b(R;\Delta)) = \psi_{(N+1)e}(F^{(N+1)e}_* \tau_b(R;\Delta)) = \ldots
\]
and so we also have
\[
\tau_b(R;\Delta, f^{{a \over p^{Ne}}{p^{Ne} - 1 \over p^e - 1}}) = \tau_b(R;\Delta, f^{{a \over p^{(N+1)e}}{p^{(N+1)e} - 1 \over p^e - 1}}) = \ldots
\]
for the same $N$.  But then notice that as $n$ increases, ${a \over p^{(n+1)e}}{p^{(n+1)e} - 1 \over p^e - 1}$ approaches $\frac{a}{p^e-1}$ from below.  Thus choose $c = {a \over p^{Ne}}{p^{Ne} - 1 \over p^e - 1}$.
\end{proof}

As our final ingredient we state a trivial extension of \cite[Proposition 3.3]{KatzmanLyubeznikZhangOnDiscretenessAndRationality} the proof of which is also verbatim the same:
\begin{proposition}
\label{propIrrationalEmpty}
 Let $C \subset [0,m] \setminus \bQ$ be a set with the following properties:
\begin{itemize}
 \item[(a)]  for all $c \in C$, there exists an $\epsilon \geq 0$ such that $(c, c+\epsilon) \cap C = \emptyset$.
 \item[(b)]  $C$ is closed in $[0, m]$.
 \item[(c)]  For some fixed integer $e$, $\{p^{ne} c\} \in C$ for all $c \in C$ and $n \geq 1$.
\end{itemize}
Then $C$ is empty.
\end{proposition}

\begin{theorem}  \label{ThmFFinitePrincipalCase}
Let $X$ be a normal $F$-finite scheme with a canonical sheaf and suppose that $\Delta$ is a $\bQ$-divisor on $\Spec R$ such that $(p^e - 1)(K_R + \Delta)$ is Cartier.  Let $\ba \subseteq \O_X$ be a locally principal ideal sheaf.  Then the set of $F$-jumping numbers of $\tau_b(X; \Delta, \ba^*)$ is a discrete set of rational numbers.
\end{theorem}
\begin{proof}
The proof follows along the lines of the proof of \cite[Theorem 3.1]{KatzmanLyubeznikZhangOnDiscretenessAndRationality}.  Without loss of generality, we may assume that $X = \Spec R$ and that $\ba = (g)$. We first note there cannot be any rational numbers that are accumulation points of $F$-jumping numbers.  To see this, suppose that $r \in \bQ$ was such an accumulation point (note this implies that $r$ is also is an $F$-jumping number).  We can then use Lemma \ref{LemmaPAlphaJumps} to remove $p$ from the denominator of $r$ and use Proposition \ref{PropLessThanRatIsOk} to obtain the contradiction. Now Proposition \ref{propIrrationalEmpty} rules out any irrational accumulation points of jumping numbers.
\end{proof}

\section{Further questions and remarks}
\label{SectionFurtherRemarks}

While in this paper we have proven discreteness and rationality of $F$-jumping numbers in essentially the same generality as is known in characteristic zero (which is to say, for normal $\bQ$-Gorenstein varieties over an algebraically closed field), a number of possible generalizations remain.

\begin{question}
Suppose that $R$ is normal, $F$-finite $\bQ$-Gorenstein but the index is divisible by $p > 0$?  Is it true that the set of $F$-jumping numbers of $\tau_b(R; \ba^t)$ are discrete and rational?
\end{question}

\begin{question}
Suppose that $R$ is $F$-finite and normal (but not necessarily $\bQ$-Gorenstein).  Is it true that the set of $F$-jumping numbers of $\tau_b(R; \ba^t)$ are discrete and rational?
\end{question}

\begin{remark}
In \cite{DeFernexHacon}, de~Fernex and Hacon define a multiplier ideal for normal non-$\bQ$-Gorenstein varieties.  It is an open question whether the jumping numbers of de~Fernex and Hacon's multiplier ideal are discrete and rational.  See \cite[Remark 4.10]{DeFernexHacon}
\end{remark}

\begin{question}
Suppose that $R$ is normal and $\bQ$-Gorenstein with index not divisible by $p > 0$.  Further suppose that $R$ is not (necessarily) $F$-finite.  Are the set of $F$-jumping numbers for $\tau_b(R, \ba^t)$ discrete and rational?  Note that one needs $R$ to be local (or at least finite type over an excellent semi-local ring) for us to know that test elements even exist.  We expect that in the local case, one can dualize the techniques from Section \ref{SectionDiscretenessAndRationalityForPrincipal} (and combine them with ideas from \cite{KatzmanLyubeznikZhangOnDiscretenessAndRationality}) to answer this question.  We do not attempt this here however.
\end{question}

\begin{remark}
One setting that we could easily generalize to is the non-normal case.  However, on non-normal varieties, the notion of $\bQ$-Gorenstein (or $\bQ$-divisors in general) may not even make sense and, at best, is much harder to work with, see for example \cite[Chapter 16]{KollarFlipsAndAbundance}.

However, inspired by the correspondence between effective $\bQ$-divisors and $\O_X$-linear maps $\phi_e : F^e_* \sL \rightarrow \O_X$ found in \ref{EqnGlobalBijection} on page \pageref{EqnGlobalBijection}, instead of triples $(X, \Delta, \ba^t)$, one should consider triples $(X, [\phi_e], \ba^t)$.  Here $[\phi_e]$ is an appropriately defined equivalence class of maps.  The equivalence relation is generated by the following two equivalences.
\begin{itemize}
\item[(1)]  $\phi_e \sim \phi_{e}'$ (note the same $e$) if there is a diagram just as the one below equation (\ref{EqnGlobalBijection}) on page \pageref{EqnGlobalBijection}.
\item[(2)]  $\phi_{ne} \sim \phi_e$ where $\phi_{ne}$ is defined as in Remark \ref{RemarkDefinitionOfComposedPInverseLinear}.
\end{itemize}
For such triples, $(X, [\phi_e], \ba^t)$, we believe that essentially all the results in this paper hold true. A generalization of this viewpoint is undertaken by the second author in \cite{Schwede.NonQGor} to study test ideals in the non $Q$-Gorenstein setting.

\end{remark}

\begin{remark}
Another way which the results of this paper can easily be generalized is to prove a general discreteness and rationality result for the $F$-jumping numbers of $\tau_b(X, \nsubseteq Q; \Delta, \ba^t)$ for $X$, $\Delta$ and $Q \in \Spec R$ satisfying the conditions of Definition \ref{DefnUltraGeneralizedTestIdeal} (at least if one assumes that $X$ is essentially of finite type over an $F$-finite field or if one assumes that $X$ is $F$-finite and $\ba$ is locally principal).
\end{remark}

\providecommand{\bysame}{\leavevmode\hbox to3em{\hrulefill}\thinspace}
\providecommand{\MR}{\relax\ifhmode\unskip\space\fi MR}
\providecommand{\MRhref}[2]{%
  \href{http://www.ams.org/mathscinet-getitem?mr=#1}{#2}
}
\providecommand{\href}[2]{#2}

\end{document}